\documentclass{article}

\usepackage{arxiv}

\usepackage[utf8]{inputenc} 
\usepackage[T1]{fontenc}    
\usepackage{hyperref}       
\usepackage{url}            
\usepackage{booktabs}       
\usepackage{amsfonts}       
\usepackage{nicefrac}       
\usepackage{microtype}      
\usepackage{lipsum}		
\usepackage{graphicx}
\usepackage{natbib}
\usepackage{doi}

\newcommand{\R}{{\rm I\!R}}

\usepackage{amsmath}
\usepackage{amssymb}
\usepackage{amsfonts}
\usepackage{bm}

\usepackage{amsthm}
\usepackage{overpic}
\usepackage{lipsum}
\usepackage{enumitem}
\usepackage{cases}
\usepackage{subcaption}
\usepackage{graphicx}

\DeclareMathOperator*{\argmax}{arg\,max}

\usepackage{ifthen}
\usepackage{graphicx,epstopdf,amssymb,amsthm}
\usepackage{ulem,color}
\usepackage{epstopdf}

\usepackage[utf8]{inputenc}
\usepackage[english]{babel}
\usepackage{caption}

\DeclareCaptionLabelFormat{lc}{\MakeLowercase{#1}~#2}
{
  \theoremstyle{plain}
  \newtheorem{assumption}{Assumption}
}

\newtheorem{theorem}{Theorem}
\newtheorem{corollary}{Corollary}
\newtheorem{remark}{Remark}
\newtheorem{lemma}{Lemma}

\newtheorem{definition}{Definition}
\newtheorem{proposition}{Proposition}
\newtheorem{example}{Example}[]
\newcommand{\edit}[1]{\textcolor{black}{#1}}

\title{Analysis and Applications of Population Flows \\ in a Networked SEIRS Epidemic Process}

\date{} 					

\author{ \href{https://orcid.org/0000-0002-2489-4411}{\includegraphics[scale=0.06]{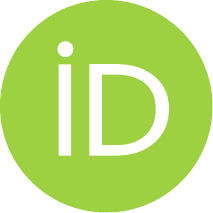}\hspace{1mm}Brooks A. Butler}\\
	Elmore Family School of Electrical and Computer Engineering\\
	Purdue University\\
	\texttt{brooksbutler@purdue.edu} \\
        \And
	\href{https://orcid.org/0000-0001-6633-7827}{\includegraphics[scale=0.06]{orcid.pdf}\hspace{1mm}Raphael Stern} \\
	Department of Civil Engineering\\
	University of Minnesota\\
	\texttt{rstern@umn.edu} \\
	\And
	\href{https://orcid.org/0000-0002-4095-7320}{\includegraphics[scale=0.06]{orcid.pdf}\hspace{1mm}Philip E. Par\'{e}}
        \thanks{This work was funded in part by the National Science Foundation, grants NSF-CNS \#2028738 (P.E.P.), NSF-CNS \#2028946 (R.S.), and NSF-ECCS \#2032258 and \#2238388 (P.E.P.).}\\
	Elmore Family School of Electrical and Computer Engineering\\
	Purdue University\\
	\texttt{philpare@purdue.edu} \\
}

\renewcommand{\headeright}{}
\renewcommand{\undertitle}{}
\renewcommand{\shorttitle}{Networked SERIS with Travel Flows}

\hypersetup{
pdftitle={Analysis and Applications of Population Flows in a Networked SEIRS Epidemic Process},
pdfsubject={q-bio.NC, q-bio.QM},
pdfauthor={Brooks A. Butler, Raphael Stern, Philip E. Par\'{e}},
pdfkeywords={Epidemic Spread, Transportation, Network Modeling},
}

\begin{document}
\maketitle

\begin{abstract}
Transportation networks play a critical part in the spread of infectious diseases between populations. In this work, we define a networked susceptible-exposed-infected-recovered epidemic process with loss of immunity over time (SEIRS) that explicitly models the flow of individuals between sub-populations, which serves as the propagating mechanism for infection. We provide sufficient conditions for local stability and instability of the healthy state of the system and show that no perturbation of population flows can change the local stability of any healthy state. We also provide sufficient conditions for the existence and uniqueness of an endemic state. We then develop tools and methods for applying our model to real-world data, including spreading parameter estimation and disease arrival time prediction, and apply them in a case study using both travel and infection data from counties in Minnesota during the first year of the COVID-19 pandemic.
\end{abstract}


\section{Introduction}
With the prevalence of transportation networks that enable rapid travel between locations globally, understanding how these transportation networks facilitate the spread of infectious diseases only continues to grow in importance. While certainly not the first globally transmitted disease, the recent COVID-19 pandemic has demonstrated the negative effect that a highly infectious disease can have on general health and safety of any population connected via travel~\cite{ruan2006effect,tatem2006global,lowe2014role,findlater2018human,hollingsworth2006will,tang2009epidemic}, as well as long-lasting impacts on both local and global economies~\cite{maital2020global,kaye2021economic,ahmad2020coronavirus,kumar2020social}. Given the growing impact and scale of human mobility as global interconnectivity continues to increase, as well as the near certainty of other pandemic-level diseases emerging in the future, it is critical that we develop a rigorous and sophisticated understanding of such phenomena in the networked sense. Thus, in this work, we aim to contribute to a deeper understanding of how transportation networks may be used to model and predict the transmission of infectious diseases over a network of connected populations.

This work builds off the deep body of literature dedicated to modeling epidemic processes using compartmental models~\cite{kermack1927contribution, brauer2008compartmental, prasse2020network, fagnani2017time, chen2019global}. Such models separate a population, or individuals in a population, into distinct compartments which represent their infection state. Simple compartmental models may have as few as two states, such as being either susceptible to a given disease or infected. For diseases with more complex behavior, more compartments may be needed to better describe the progression of a population through the epidemic process of the disease. In this work, we focus on a four-compartment model which separates individuals into one of four states, namely: susceptible, exposed, infected, and recovered. Further, we allow for a rate at which a population or individual may lose immunity to the given disease, and therefore may become susceptible again. This model, abbreviated to the SEIRS model, allows for the modeling of a disease that has a period in which an individual may be exposed and will become infectious, but is not immediately apparent. Additionally, the recovered state and eventual loss of immunity model a disease that can evolve over time, but may allow for temporary immunity after an initial recovery. This model is chosen, in part, to better emulate what is now known of the behavior of COVID-19 and its variants \cite{xiang2021covid,rahmandad2021behavioral,wang2020evaluating,chen2020time}. Compartmental models alone, however, do not account for the transmission of a disease in a networked environment where a general population can be divided into subpopulations connected via various transportation networks \cite{nowzari2016analysis,qian2021connecting,salehi2015spreading,qu2017sis,kang2018spreading,nadini2018epidemic,li2021modeling,yang2012epidemic,hackl2019epidemic,pare2020modeling}. Thus, an increase in scope is needed to account for this important behavioral complexity in disease transmission.

Previous work that incorporates networked population flows in an epidemic process model includes analysis of \edit{both networked SIS and SIR models} with flows \cite{ye2020network, yin2020novel}  as well as using a networked SIR model with flows to predict arrival times for various epidemics using global flight data \cite{brockmann2013hidden}, where both models are developed in continuous time. \edit{Additionally, a data-driven approach to modeling the effect of population flow on epidemic spreading is performed in \cite{di2021data}, which uses data from lockdowns in Italy during the initial spreading of COVID-19. Additional work on epidemic spreading with population flow has been performed in modeling multi-city epidemics \cite{arino2003multi, liu2013transmission, knipl2016stability, pujari2020multi} and in the modeling of disease spread in metapopulations \cite{arino2006disease, wang2014spatial, shao2022epidemic, zhu2021allocating}.} We contribute to the development of such networked models by including the exposed state in our model formulation as well as a loss of immunity. In our previous work, we consider capturing the effect of transportation on the spread of COVID-19 using a networked discrete-time SEIR model \cite{vrabac2021capturing, butler2021effect}; however, the key distinction in this work is that, in addition to being a continuous-time model, infection propagation over the network is modeled by the relocation of infected individuals to other sub-populations rather than assuming direct contact and intermingling between sub-populations. Further, we allow recovered individuals to experience a loss of immunity over time, which creates the possibility of an endemic state for the epidemic process.

\noindent In this paper, we make the following contributions:
\begin{enumerate}
    \item We define a networked SEIRS model that incorporates travel flows as a means of disease propagation
    \item We provide conditions for local stability and instability of the healthy state and prove that perturbations to population flows cannot change the stability of the healthy state
    \item We provide conditions under which a unique endemic state will exist
    \edit{
    \item We present methods for estimating the spreading parameters of our model using both infection and travel data as well as a method for predicting the arrival time of a disease using travel data to compute the effective distance between nodes
    }
    \item We demonstrate how these tools can be applied to our model using real-world \edit{mobility data from Minnesota during the COVID-19 pandemic}.
\end{enumerate}
We now provide the necessary definitions of notation and previous analytical results used in our model construction and analysis sections, respectively.

\edit{\section{Preliminaries}}

We denote the set of real numbers and positive real numbers as $\mathbb{R}$ and $\mathbb{R}_{>0}$, respectively. For any positive integer $n$, we have $[n] =\{1,2,...,n \}$. 
A diagonal matrix is denoted as diag$(\cdot)$. The transpose of a vector $x\in \mathbb{R}^n$ is $x^\top$. We use $\mathbf{0}$ and $\mathbf{1}$ to denote the vectors of the appropriate size whose entries all equal 0 and 1, respectively. We let $\mathcal{G} =(\mathbb{V},\mathbb{E},\mathbb{W})$ denote a weighted directed graph  
where $\mathbb{V} = \{ v_1, v_2,..., v_n\}$ is the set of nodes, $\mathbb{E} \subseteq \mathbb{V}\times \mathbb{V}$ is the set of edges, and 
$\mathbb{W}:\mathbb{E}\rightarrow \mathbb{R}_{>0}$
maps to 
the real valued edge weights on each edge. We denote the configuration of edges in a directed graph at time $t$ as $\mathcal{G}(t)=(\mathbb{V},\mathbb{E}(t),\mathbb{W})$, where 
$\mathbb{E}(t)$ denotes the set of edges at time $t$. 
Furthermore, we denote $\cup_{t \geq 0}\mathbb{E}(t)$ as the union of all non-zero edge configurations on a graph for all $t \geq 0$. We define a graph $\mathcal{G}$ as being strongly connected if there is a path consisting of nonzero edge weights from every node to every other node in the graph.

For a complex number $x$ we use $|x|$ and $\text{Re}(x)$ to denote its magnitude and real part, respectively. For a real square matrix $M$, we use $\rho(M)$ to denote its spectral radius and $s(M)$ to denote the largest real part amongst its eigenvalues, i.e.,
\begin{align*}
    \rho(M) &= \max \{|\lambda|: \lambda \in \sigma (M) \}, \\
    s(M) &= \max \{ \text{Re}(\lambda): \lambda \in \sigma(M) \},
\end{align*}
where $\sigma(M)$ denotes the spectrum of $M$.

A real square matrix is called \textit{Metzler} if its off-diagonal entries are all non-negative. Thus, any non-negative matrix is Metzler. 
\begin{lemma} \label{lem:metzler}
For any matrix $M$ and any real number $\phi$, if $A:=M-\phi I$, then $\sigma(M)=\sigma(A)+\phi$.
\end{lemma}

\noindent The following results from Chapter 2 of \cite{varga1962iterative} for non-negative matrices, which also hold for Metzler matrices by Lemma \ref{lem:metzler}, with $\phi = \min \{ 0, m_{11}, \dots, m_{nn} \}$, will be used in the subsequent analysis.

\begin{lemma}
\label{lem:stricly_pos_vec_M}
(Lemma 2.3 in \cite{varga1962iterative}) Suppose that $M$ is an irreducible Metzler matrix. Then, $s(M)$ is a simple eigenvalue of $M$ and there exists a unique (up to scalar multiple) vector $x \gg \mathbf{0}$ such that $Mx=s(M)x$.
\end{lemma}

\begin{proposition}
\label{prop:spec_rad}
Suppose that $\Lambda$ is a negative diagonal matrix in $\mathbb{R}^{n \times n}$ and $N$ is an irreducible non-negative matrix in $\mathbb{R}^{n \times n}$. Let $M = \Lambda+N$. Then, $s(M) < 0$ if and only if $\rho(-\Lambda^{-1}N)<1$, $s(M) = 0$ if and only if $\rho(-\Lambda^{-1}N)=1$, and $s(M) > 0$ if and only if $\rho(-\Lambda^{-1}N)>1$. 
\end{proposition}

\begin{lemma}
\label{lem:diag_dominant}
(Levy–Desplanques Theorem) A strictly diagonally dominant matrix is non-singular. In other words, let $A \in C^{n \times n}$ be a matrix satisfying the property
\begin{equation*}
    |a_{ii}| > \sum_{j \neq i} |a_{ij}|, \,\, \forall i;
\end{equation*}
\noindent then $\det (A) \neq 0$.
\end{lemma}
\begin{lemma}
\label{lem:eig_pert}
(Observation 6.3.1 in \cite{horn2012matrix}) Let $M \in \R^{n \times n}$ be diagonalizable with $M = S \Lambda S^{-1}$ and $\Lambda = diag(\lambda_1, \dots, \lambda_n)$ and let $E \in \R^{n \times n}$. If $\hat{\lambda}$ is an eigenvalue of $M+E$, then there exists some eigenvalue $\lambda_i$ of $M$ for which
\begin{align*}
    |\hat{\lambda} - \lambda_i| \leq {||S||}_{\infty}{||S^{-1}||}_{\infty}{||E||}_{\infty} = \kappa_{\infty}(S){||E||}_{\infty},
\end{align*}
where $\kappa_{\infty}(S)$ denotes the condition number with respect to the infinity matrix norm ${|| \cdot ||}_{\infty}$.
\end{lemma}

\section{Network Flows Model Definition} \label{sec:model}
We now present a networked SEIRS model incorporating the population flow of individuals between subpopulations.
First, consider a group of $n$ sub-populations in a graph, where each sub-population $i\in [n]$ is represented by a node
in the graph $\mathcal{G}$. We use the SEIRS model to describe how susceptible individuals in sub-population $i$ become exposed, infected, recover, and gradually lose immunity 
as the result
of an infectious disease \cite{shu2012global}. 
We begin with defining the SEIRS model behavior without graph connections 
for each sub-population $i\in [n]$. Let $S_i$, $E_i$, $I_i$, and $R_i$ represent the number of susceptible, exposed, infected, and recovered individuals in sub-population $i$, respectively, and their dynamics evolve as

\begin{subequations}
\label{eq:seir_standard}
\begin{align}
        \dot{S}_i(t) &= \alpha_i R_i(t) -\beta_i \frac{I_i(t)}{N_i} S_i(t) \\
        \dot{E}_i(t) &= \beta_i \frac{I_i(t)}{N_i} S_i(t) - \sigma_i E_i(t) \\
        \dot{I}_i(t) &= \sigma_i E_i(t) - \delta_i I_i(t) \\
        \dot{R}_i(t) &= \delta_i I_i(t) - \alpha_i R_i(t),
\end{align}
\end{subequations}

\noindent where $\beta_i$ is the infection rate, $\sigma_i$ is \edit{the} transition rate from exposed to infected, $\delta_i$ is the healing rate, and $\alpha_i$ is the rate of immunity loss.
Although the \edit{state} variables, except population, will continue to vary with time, we remove the time-dependence notation for convenience and ease of reading from this point forward.

To account for the flow of individuals between sub-populations 
we expand the model
in~\eqref{eq:seir_standard}:
\begin{subequations}
\label{eq:seir_flows_indv}
\begin{align}
    \dot{S}_i &= -\beta_i \frac{I_i}{N_i} S_i + \edit{\alpha_i} R_i +\sum_{j \neq i} \left(F_{ij}\frac{S_j}{N_j}-F_{ji}\frac{S_i}{N_i} \right) \\
    \dot{E}_i &= \beta_i \frac{I_i}{N_i} S_i - \sigma_iE_i+ \sum_{j \neq i} \left(F_{ij}\frac{E_j}{N_j}-F_{ji}\frac{E_i}{N_i} \right) \\
    \dot{I}_i &= \sigma_iE_i - \delta_i I_i+ \sum_{j \neq i} \left(F_{ij}\frac{I_j}{N_j}-F_{ji}\frac{I_i}{N_i} \right) \\        
    \dot{R}_i &= \delta_i I_i - \edit{\alpha_i} R_i + \sum_{j \neq i} \left(F_{ij}\frac{R_j}{N_j}-F_{ji}\frac{R_i}{N_i} \right),
\end{align}
\end{subequations}

\noindent where $F_{ij}$ represents the number of individuals flowing from sub-population $j$  to $i$, with $F_{ii}=0$. By making a substitution of variables where $s_i = S_i/N_i, e_i = E_i/N_i, x_i = I_i/N_i, r_i = R_i/N_i$ we can model the proportion of individuals 
as follows
\begin{subequations}
\label{eq:seir_flows_proportion}
\begin{align}
    \dot{s}_i &=\alpha_i r_i -\beta_i x_i s_i +\frac{1}{N_i}\sum_{j \neq i} \left(F_{ij}s_j-F_{ji}s_i \right) \\
    \dot{e}_i &= \beta_i x_i s_i - \sigma_i e_i+ \frac{1}{N_i}\sum_{j \neq i} \left(F_{ij}e_j-F_{ji}e_i \right) \\
    \dot{x}_i &= \sigma_i e_i - \delta_i x_i+ \frac{1}{N_i}\sum_{j \neq i} \left(F_{ij}x_j-F_{ji}x_i \right) \\        
    \dot{r}_i &= \delta_i x_i- \alpha_i r_i + \frac{1}{N_i}\sum_{j \neq i} \left(F_{ij}r_j-F_{ji}r_i \right), 
\end{align}
\end{subequations}
\noindent where $s_i+e_i+x_i+r_i=1$. Note that both \eqref{eq:seir_flows_indv} and \eqref{eq:seir_flows_proportion} 
assume the subpopulations are well mixed and that the likelihood of an individual traveling is independent of their infectious state, that is, whether they are susceptible, exposed, infected, or recovered. We can compute the number of individuals flowing from sub-population $j$ to $i$ as     
$F_{ij} = \gamma_j w_{ij} N_j$,
where $\gamma_j$ is the \edit{proportion of the total population} flowing out of node $j$ computed as
\edit{$\gamma_j = \frac{\sum_{i \neq j}F_{ij}}{N_j}$},
and $w_{ij}$ is the proportion of traveling individuals flowing from sub-population $j$ to $i$ computed as
$w_{ij} = \frac{F_{ij}}{\sum_{l \neq j}F_{lj}}$,
with $w_{ii} = 0$. 
Thus, we can further derive the dynamics for the susceptible proportion at sub-population $i$ as

\begin{align*}
    \dot{s}_i &= \alpha_i r_i -\beta_i x_i s_i +\frac{1}{N_i}\sum_{j \neq i} \left(F_{ij}s_j-F_{ji}s_i \right) \\
     &= \alpha_i r_i -\beta_i x_i s_i +\frac{1}{N_i}\sum_{j \neq i} \left(\gamma_j w_{ij} N_j s_j-\gamma_i w_{ji} N_i s_i \right) \\
     &= \alpha_i r_i -\beta_i x_i s_i +\sum_{j \neq i} \left(\frac{N_j}{N_i} w_{ij}\gamma_j s_j- w_{ji} \gamma_i s_i \right).
\end{align*}
\noindent 
Using the fact that $\sum_{j \neq i}w_{ji} = 1$, we have that
\begin{equation*}
     \dot{s}_i = \alpha_i r_i -(\beta_i x_i+\gamma_i) s_i +\sum_{j \neq i} \frac{N_j}{N_i} w_{ij}\gamma_j s_j.
\end{equation*}
\noindent
By similar derivations we can rewrite \eqref{eq:seir_flows_proportion} as
\begin{subequations}
\label{eq:flows_ind_node_cont}
\begin{align}
     \dot{s}_i &= \alpha_i r_i -(\beta_i x_i+\gamma_i) s_i +\sum_{j \neq i} \frac{N_j}{N_i} w_{ij}\gamma_j s_j \label{eq:flows_ind_node_cont_s} \\
     \dot{e}_i &= \beta_i x_i s_i - (\sigma_i+\gamma_i)e_i + \sum_{j \neq i} \frac{N_j}{N_i} w_{ij}\gamma_j e_j \label{eq:flows_ind_node_cont_e} \\
     \dot{x}_i &= \sigma_i e_i - (\delta_i+\gamma_i)x_i + \sum_{j \neq i} \frac{N_j}{N_i} w_{ij}\gamma_j x_j \label{eq:flows_ind_node_cont_x} \\
     \dot{r}_i &= \delta_i x_i -(\gamma_i+\alpha_i)r_i +\sum_{j \neq i} \frac{N_j}{N_i} w_{ij}\gamma_j r_j \label{eq:flows_ind_node_cont_r}
\end{align}
\end{subequations}

\begin{figure}
    \centering
    \begin{overpic} [width=.6\columnwidth]{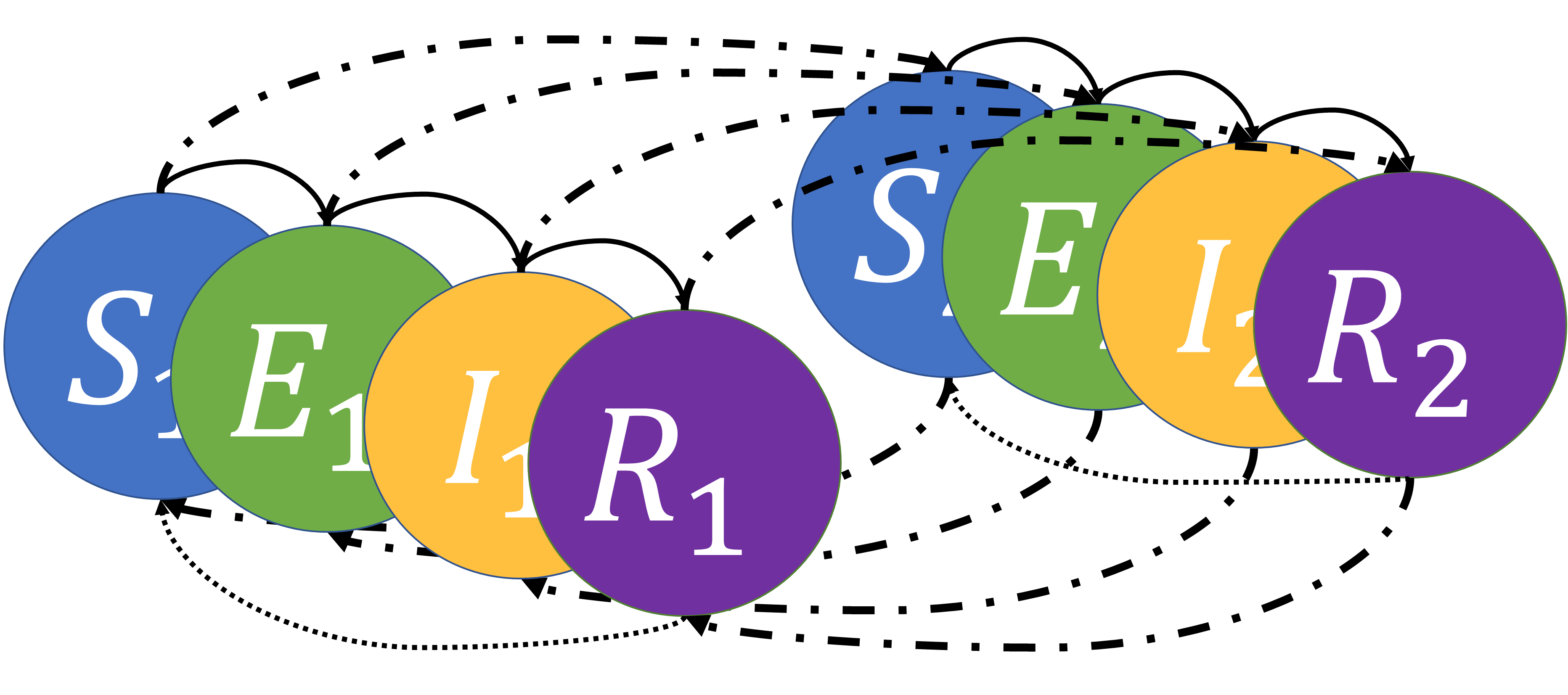}
    \put(40,43.5){\small $\gamma_1$}
    \put(18,35){\small $\beta_1$}
    \put(29,32){\small $\sigma_1$}
    \put(40,29){\small $\delta_1$}
    \put(22,5){\small $\alpha_1$}
    \put(65,43){\small $\beta_2$}
    \put(74,41){\small $\sigma_2$}
    \put(85,38){\small $\delta_2$}
    \put(70.5,15){\small $\alpha_2$}
    \put(63,0){\small $\gamma_2$}
    \end{overpic}
    \caption{\edit{An illustration of a two-node SIERS network with travel flows as defined by \eqref{eq:flows_ind_node_cont}.}}
    \label{fig:two_node_diagram}
\end{figure}
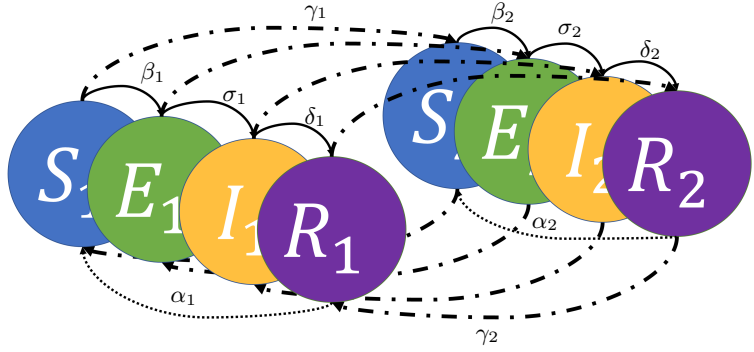

\noindent \edit{where an illustration of this model for two nodes is shown in Figure~\ref{fig:two_node_diagram}. Note that our model does not allow for state transitions to occur during travel, i.e., we do not model individuals becoming infected as a result of their travel. Rather, individuals will change their state in the epidemic process at the rate defined at each node. Of course, this may not always be the case in reality, as some methods of public transportation may facilitate infection spreading for those sharing a confined space with inadequate ventilation. Therefore, this study implicitly focuses on more long-range modes of transportation that transfer individuals between more distinctly separated populations (such as travel via flights or long-rage commuting) where traveling individuals may not be interacting with other travelers, or where there is a sufficient amount of ventilation such that the spread of the disease is significantly reduced during transit \cite{hu2021flights}. The inclusion of infections occurring on the edges of transportation networks is left as a direction for future work.} 

For the model in \eqref{eq:flows_ind_node_cont} to be well-defined, we required the following assumptions.

\begin{assumption}
\label{assume:equal_flows}
Let $\sum_{i \neq j} F_{ji} = \sum_{i \neq j}F_{ij}$ for all $i,j \in [n]$.
\end{assumption}

\noindent This assumption requires that the total flow of individuals into a given sub-population must be equal to the total flow out, which naturally follows from the assumption that the population at each node remains constant. 
Furthermore, we impose the following assumption on the model parameters and initial conditions.

\begin{assumption}
\label{assume:parameters}
Let $\alpha_i, \beta_i, \sigma_i, \delta_i \in \mathbb{R}_{>0}$, $s_i(t_0),e_i(t_0),x_i(t_0),r_i(t_0) \in [0,1]$, and $s_i(t_0)+e_i(t_0)+x_i(t_0)+r_i(t_0) =1$ for all $i \in [n]$ and $t_0 \in \mathbb{R}_{\geq 0}$. 
\end{assumption}
\noindent
Under these assumptions, we can show that the model will always remain well-defined.

\begin{lemma}
\label{lem:welldefined}
Let  Assumptions \ref{assume:equal_flows}-\ref{assume:parameters} hold, then $s_i(t),e_i(t),x_i(t),r_i(t) \in [0,1]$, and $s_i(t)+e_i(t)+x_i(t)+r_i(t) =1$ for all $i \in [n]$ and for all $t \in \mathbb{R}_{\geq 0}$.
\end{lemma}
\begin{proof}
First, by taking the sum of the system dynamics as defined in \eqref{eq:seir_flows_proportion} we have  
\begin{align}\label{eq:sum_der_eq_0}
    \dot{s}_i(t)+\dot{e}_i(t)+\dot{x}_i(t)+\dot{r}_i(t) = 0,
\end{align}
by Assumption \ref{assume:equal_flows} at every time $t$. Thus, by integrating 
\eqref{eq:sum_der_eq_0} from $t_0$ to $t$ we have
\begin{align*}
    \int_{t_0}^t \dot{s}_i(\tau)+\dot{e}_i(\tau)+\dot{x}_i(\tau)+\dot{r}_i(\tau) d\tau = 0,\\
    \implies s_i(t)+e_i(t)+x_i(t)+r_i(t) =1 , t \in \mathbb{R}_{\geq 0}
\end{align*}
by Assumption \ref{assume:parameters}. To show that $s_i(t),e_i(t),x_i(t),r_i(t) \in [0,1]$, we consider the system  as each state approaches zero. Note that $\lim_{q_i(t) \rightarrow 0}\dot{q}_i(t) \geq 0$ where $q_i \in \{s_i,e_i,x_i,r_i\}$. Thus, by the assumption that the initial states are non-negative, we have that  $q_i(t) \geq 0, \forall t \geq 0$. Further, since we have shown that $s_i(t)+e_i(t)+x_i(t)+r_i(t) =1$ it follows directly that $s_i(t),e_i(t),x_i(t),r_i(t) \in [0,1], \forall t \geq 0$ and $\forall i \in [n]$.
\end{proof}

For the purpose of the forthcoming analysis, we also express the system in \eqref{eq:flows_ind_node_cont} in matrix form. Let $\Phi = N^{-1}W \Gamma N$, where $N = diag(N_i)$, $W \in \mathbb{R}^{n \times n}$ is the matrix comprised of the entries $w_{ij}$, and $\Gamma = diag(\gamma_i)$. The vectorized equations then become
\begin{subequations}
\label{eq:matrix_eq}
\begin{align}
     \dot{s} &= A r -(B X(t)+\Gamma) s +\Phi s  \\
     \dot{e} &= B X(t) s - (\Sigma+\Gamma)e + \Phi e \\
     \dot{x} &= \Sigma e - (D+\Gamma)x + \Phi x \\
     \dot{r} &= D x -(A+\Gamma)r +\Phi r,
\end{align}
\end{subequations}
where $A = diag(\alpha_i)$, $B = diag(\beta_i)$, $\Sigma = diag(\sigma_i)$, $D= diag(\delta_i)$, $X(t) = diag(x_i)$, $s = [s_1, \dots , s_n]^\top$, $e = [e_1, \dots , e_n]^\top$, $x = [x_1, \dots , x_n]^\top$, and $r = [r_1, \dots , r_n]^\top$. 
We further construct the matrix
\begin{equation}
\footnotesize
\underbrace{
    \begin{bmatrix}
    -BX(t)-\Gamma+\Phi & 0 & 0 & A \\
    BX(t) & -\Sigma-\Gamma+\Phi & 0 & 0 \\
    0 & \Sigma & -D-\Gamma+\Phi & 0 \\
    0 & 0 & D & -A-\Gamma+\Phi
    \end{bmatrix}
    }_{H(t)},
\end{equation}
where
\begin{equation}
    H(t) = -Q(t)+M(t),
\end{equation}
with
\begin{equation}
    Q(t)=
    \begin{bmatrix}
    \edit{BX(t)}+\Gamma &  &  &  \\
    & \Sigma+\Gamma &  &  \\
    &  & D+\Gamma & \\
    &  &  & A+\Gamma
    \end{bmatrix},
\end{equation}
\begin{align}
    M(t) &=
    \begin{bmatrix}
    \Phi & 0 & 0 & A \\
    BX(t) & \Phi & 0 & 0 \\
    0 & \Sigma & \Phi & 0 \\
    0 & 0 & D & \Phi
    \end{bmatrix},
\end{align}
and let the complete state vector be defined as
\begin{equation}
\label{eq:full_state_vec}
    z = \begin{bmatrix} s^\top & e^\top & x^\top & r^\top \end{bmatrix}^\top.
\end{equation}
Thus, 
we can describe the state dynamics as
\begin{equation}
\label{eq:matrix_form}
    \dot{z} = \big( -Q(t) + M(t)\big) z,
\end{equation}
where $Q(t)$ is a positive diagonal matrix and $M(t)$ is a non-negative matrix $\forall t \geq 0$. 
Note that the time-dependence of $Q(t)$ and $M(t)$ comes from the fact that they are functions of $X(t)$, where the indication of time dependence is left for emphasis of this fact.

\begin{definition}
A graph $\mathcal{G}(t)= (\mathbb{V},\cup_{t\geq 0}\mathbb{E}(t),\mathbb{W})$ for $t \in \mathbb{R}_{\geq 0}$ is \emph{$K$-strongly connected} if there exist some bound $K$ such that   
$(\mathbb{V},\cup_{j=t}^{t+K-1} \mathbb{E}(j), \mathbb{W})$
is strongly connected, for all $t \in \mathbb{R}_{\geq 0}$.
\end{definition}
\begin{assumption} \label{assume:strongly_connected_bounded_comms}
Let the graph $\mathcal{G}(t) = (\mathbb{V},\cup_{t\geq 0}\mathbb{E}(t),\mathbb{W})$, 
where 
$\mathbb{W}:\mathbb{E}(t)\rightarrow \mathbb{R}_{>0}$ is defined by
$w_{ij}(t)$, be $K$-strongly connected.
\end{assumption}

\begin{lemma}
\label{lem:irreducible}
Given Assumptions \ref{assume:parameters}-\ref{assume:strongly_connected_bounded_comms}, if $\exists t_0 \geq 0$ and $\exists i \in [n]$ such that $x_i(t_0) > 0$, the matrix $M(t_0)$ is irreducible.
\end{lemma}
\begin{proof}
Assume, by contradiction, that $M(t_0)$ is reducible. Then $M(t_0)-Q(t_0)$ will also be reducible, as $Q(t_0)$ is a diagonal matrix. Thus, the graph described by the adjacency matrix in \eqref{eq:matrix_form} cannot be strongly connected. However, we can partition the graph into $n$ strongly connected components where each component describes the evolution of the epidemic in each subpopulation. Consider the case where $n=1$, which reduces to a non-networked SEIRS model. By Assumption \ref{assume:parameters}, this system is strongly connected in the sense of the epidemic states given $x_1(t_0) > 0$ where
\begin{align*}
    \begin{bmatrix}
     \dot{s}_1 \\ \dot{e}_1 \\ \dot{x}_1 \\ \dot{r}_1
    \end{bmatrix}
    =
    \begin{bmatrix}
     -\beta_1 x_1(t_0) & 0 & 0 & \alpha_1 \\
     \beta_1 x_1(t_0) & -\sigma_1 & 0 & 0 \\
     0 & \sigma_1 & -\delta_1 & 0 \\
     0 & 0 & \delta_1 & -\alpha_1
    \end{bmatrix}
    \begin{bmatrix}
     s_1 \\ e_1 \\ x_1 \\ r_1
    \end{bmatrix}.
\end{align*}
For $n>1$, the network connections are defined by the total flows between each sub-population.  
Thus, for the system to be reducible, there must exist a sub-population that is unreachable via travel from at least one other sub-population in the graph, which contradicts Assumption \ref{assume:strongly_connected_bounded_comms}. Therefore, $M(t_0)$ must be irreducible.
\end{proof}
Note that the irreducibulity of $M(t)$ will play an important role in proving the existence of an endemic equilibrium, which will be proven in Section~\ref{sec:analysis}. 

\section{Network Flows Model Analysis} \label{sec:analysis}
In this section, we analyze the fundamental behaviors of our model, which is divided into the analysis of the healthy and the endemic states of the system in \eqref{eq:flows_ind_node_cont}. 

\subsection{Healthy State Analysis}
We first consider the existence of a disease-free, or healthy state of the system and when the stability of the healthy state can be ensured. We then examine the effect of manipulating travel flows on the stability of the healthy state and prove that there is no permissible perturbation of flows that will change the local stability of the healthy state for any system.
\begin{proposition} \label{prop:healthy_eq_exs}
The healthy state, $(s^*,e^*,x^*,r^*)=(\mathbf{1},\mathbf{0},\mathbf{0},\mathbf{0})$, is always an equilibrium of the system.
\end{proposition}
\begin{proof}
Substituting $(\mathbf{1},\mathbf{0},\mathbf{0},\mathbf{0})$ into \eqref{eq:seir_flows_proportion} yields
\begin{align*}
     \dot{s}_i = \frac{1}{N_i}\sum_{j \neq i} \left(F_{ij}-F_{ji} \right),
     \,\,  \dot{e}_i = 0 ,
     \,\, \dot{x}_i = 0 ,
     \,\, \dot{r}_i = 0
\end{align*}
for all $i \in  [n]$, where $\sum_{j \neq i} \left(F_{ij}-F_{ji} \right) = 0$ by Assumption~\ref{assume:strongly_connected_bounded_comms}. Thus, $(s^*,e^*,x^*,r^*)=(\mathbf{1},\mathbf{0},\mathbf{0},\mathbf{0})$ is always an equilibrium of the system given Assumptions \ref{assume:equal_flows} and \ref{assume:parameters}.
\end{proof}

We now consider the stability of the healthy state of the system and the conditions needed to reach this equilibrium. 

First, since $s_i+e_i+x_i+r_i = 1$ for all $i\in [n]$, we can rewrite the dynamics of \eqref{eq:matrix_eq} with respect to the exposed, infected, and recovered states as

\footnotesize
\begin{align} \label{eq:exr_matrix_dynamics}
    \begin{bmatrix}
     \dot{e} \\ \dot{x} \\ \dot{r}
    \end{bmatrix}
    =
    \underbrace{
    \begin{bmatrix}
    -\Sigma-\Gamma+\Phi & B(I - E - X - R) & 0 \\
    \Sigma & -D-\Gamma+\Phi & 0 \\
    0 & D & -A-\Gamma+\Phi
    \end{bmatrix}
    }_{\bar{H}}
    \begin{bmatrix}
     e \\ x \\ r
    \end{bmatrix},
\end{align}
\normalsize

\noindent
where $E = diag(e_i)$, $X = diag(x_i)$, and $R = diag(r_i)$. 
The Jacobian matrix of \eqref{eq:exr_matrix_dynamics} 
evaluated at an equilibrium $(x^*,e^*,r^*)$, denoted as $J_{(e^*, x^*,r^*)}$, is

\footnotesize
\begin{align}
\label{eq:jacobian}
    \begin{bmatrix}
     -\Sigma-\Gamma+\Phi - BX^* & B(I-E^*-2X^*-R^*) & -BX^* \\
    \Sigma & -D-\Gamma+\Phi & 0 \\
    0 & D & -A-\Gamma+\Phi
    \end{bmatrix}.
\end{align}
\normalsize
Further, we define the following matrix, the top 2x2 diagonal block of $J$ evaluated at $(\mathbf{0},\mathbf{0},\mathbf{0})$,
\begin{equation}
    U = 
    \begin{bmatrix}
     -\Sigma-\Gamma + \Phi & B \\
     \Sigma & -D - \Gamma + \Phi
    \end{bmatrix},
\end{equation}
which yields the following conditions for the local stability and instability of the healthy state.
\begin{proposition} \label{prop:healthy_stable}
If $s(U) < 0$,
then the healthy state of \eqref{eq:matrix_eq} is locally exponentially stable.
\end{proposition}
\begin{proof}
Evaluating \eqref{eq:jacobian} at the healthy state equilibrium yields
\begin{equation}
\label{eq:jacobian_eval_healthy}
    J_{(\mathbf{1},\mathbf{0},\mathbf{0},\mathbf{0})}
    = 
    \begin{bmatrix}
     -\Sigma-\Gamma+\Phi & B & 0 \\
    \Sigma & -D-\Gamma+\Phi & 0 \\
    0 & D & -A-\Gamma+\Phi
    \end{bmatrix}.
\end{equation}
Since \eqref{eq:jacobian_eval_healthy} is block lower triangular, we can examine the eigenvalues of each block separately. By Assumption~\ref{assume:equal_flows} we have that $\gamma_i = \sum_{j \neq i} \phi_{ij}$, thus by the Gershgorin Circle Theorem we have that $s(-A-\Gamma-\Phi) < 0$. The upper block is equal to $U$, which is negative definite by assumption. 
Thus, we have that $s(J_{(\mathbf{1},\mathbf{0},\mathbf{0},\mathbf{0})}) < 0$, which satisfies Lyapunov's indirect method for determining the local stability of the healthy state.
\end{proof}

\begin{corollary}
If $s(U) > 0$, then the healthy state of \eqref{eq:matrix_eq} is unstable.
\end{corollary}
\begin{proof}
This result follows by the same method as the proof of Proposition~\ref{prop:healthy_stable}, which leverages \eqref{eq:jacobian} evaluated at the healthy state, shown in \eqref{eq:jacobian_eval_healthy}, which is block lower triangular. Since the upper block is equal to $U$, if $s(U) > 0$ then $s(J_{(\mathbf{1},\mathbf{0},\mathbf{0},\mathbf{0})}) > 0$. Thus, by Lyapunov's indirect method, the healthy state is unstable.
\end{proof}
\noindent We now discuss the local stability of the healthy state for \eqref{eq:exr_matrix_dynamics} with respect to changes in the rate of flow between nodes.
\begin{theorem}
\label{thm:flows_eig_Vals}
If $J_{(\mathbf{1},\mathbf{0},\mathbf{0},\mathbf{0})}$, as defined in \eqref{eq:jacobian_eval_healthy}, 
has distinct eigenvalues 
then, given Assumptions \ref{assume:equal_flows} and \ref{assume:parameters}, there exists no permissible perturbation to the population flows $\gamma_i$, for $i \in [n]$, that will change the local stability of the healthy state of \eqref{eq:matrix_eq}. 
\end{theorem}

\begin{proof}
Let $\gamma_i'$ be some perturbation to $\gamma_i, i \in  [n]$ such that Assumption \ref{assume:equal_flows} is satisfied and $\gamma'_i=\gamma_i + \theta_i$. We can then define the perturbed system dynamics of \eqref{eq:exr_matrix_dynamics} with the matrix

\footnotesize
\begin{equation}
    \bar{H}' = 
    \begin{bmatrix}
     -\Sigma-\Gamma'+\Phi' & B(I-E-X-R) & 0 \\
    \Sigma & -D-\Gamma'+\Phi' & 0 \\
    0 & D & -A-\Gamma'+\Phi'
    \end{bmatrix},
\end{equation}
\normalsize
where $\Gamma' = \Gamma + \Theta$, $\Theta = diag(\theta_i)$, and $\Phi' = N^{-1}W(\Gamma + \Theta)N$. We can then separate the perturbed system into
\begin{equation}
\label{eq:H_prime}
    \bar{H}'= \bar{H} + \bar{E},
\end{equation}
where $\bar{H}$ is the original system dynamics defined in \eqref{eq:exr_matrix_dynamics} and
\small
\begin{equation}
    \bar{E} = 
    \begin{bmatrix}
     N^{-1}W \Theta N - \Theta & 0 & 0 \\
     0 & N^{-1}W \Theta N - \Theta & 0 \\
     0 & 0 & N^{-1} W \Theta N - \Theta
    \end{bmatrix}.
\end{equation}
\normalsize
Taking the Jacobian of \eqref{eq:H_prime} yields
\begin{equation*}
    J'_{(e^*, x^*,r^*)} = J_{(e^*, x^*,r^*)} +\bar{E}.
\end{equation*}
We consider the healthy state equilibrium, which exists for all allowed parameters of our system by Proposition~\ref{prop:healthy_eq_exs}. 
Since the eigenvalues of $J_{{(\mathbf{1},\mathbf{0},\mathbf{0},\mathbf{0})}}$ are distinct, we have that $J_{(\mathbf{1},\mathbf{0},\mathbf{0},\mathbf{0})}$ is diagonalizable by some matrix $\bar{S}$, where $J_{(\mathbf{1},\mathbf{0},\mathbf{0},\mathbf{0})} = \bar{S} \Lambda \bar{S}^{-1}$ and $\Lambda = diag(\lambda_i)$. Thus, by Lemma~\ref{lem:eig_pert} we have that if $\lambda'$ is an eigenvalue of $J'$, then there exists some eigenvalue $\lambda_i$ of $J$ such that
\begin{equation}
    |\lambda' - \lambda_i| \leq \kappa(\bar{S})_{\infty} {|| \bar{E} ||}_{\infty}.
\end{equation}
However, by Assumption \ref{assume:equal_flows} we have that $\theta_i = \sum_{j \neq i} \frac{N_j}{N_i} w_{ij} \theta_j$. Thus,  ${|| \bar{E} ||}_{\infty} = 0$ and for any eigenvalue $\lambda'$ of $J'$ there exists some eigenvalue $\lambda_i, i \in [n]$ such that $\lambda' = \lambda_i$. Therefore, no perturbation of the travel flows that maintains Assumption \ref{assume:equal_flows}, and thus keeps the model well-defined, can change the local stability of the healthy state of \eqref{eq:matrix_eq}. 
\end{proof}

\subsection{Endemic State Analysis}
We now present conditions for when an endemic state of the system, where $z^* \gg \mathbf{0}$, will exist. For clarity, we define the following matrices explicitly as a function of the full system state
\footnotesize
\begin{align}
    Q^*(z^*)&=
    \begin{bmatrix}
    BX^*(z^*)+\Gamma & 0 & 0 & 0 \\
    0 & \Sigma+\Gamma & 0 & 0 \\
    0 & 0 & D+\Gamma & 0 \\
    0 & 0 & 0 & A+\Gamma
    \end{bmatrix} \\
    M^*(z^*) &=
    \begin{bmatrix}
    \Phi & 0 & 0 & A \\
    BX^*(z^*) & \Phi & 0 & 0 \\
    0 & \Sigma & \Phi & 0 \\
    0 & 0 & D & \Phi
    \end{bmatrix},
\end{align}
\normalsize
where
$X^*(z^*) = diag \left( 
    \begin{bmatrix}
     0 & 0 & I_n & 0
    \end{bmatrix}
    z^* \right).$
\begin{theorem} \label{thm:endemic_eq}
If $\inf_{t \geq t_0}s(-Q(t) + M(t)) > 0$, then there exists an endemic equilibrium such that $z^* \gg \edit{\mathbf{0}}$. 
\end{theorem}

\begin{proof}
Define the continuous map $f:(0,1]^{4n} \rightarrow [0,b]^{4n}$, where $b \in \mathbb{R}_{\geq 1}$, given by
\begin{equation}
\label{eq:equilibrium_eq}
    f(z) = Q^*(z)^{-1}M^*(z) z.
\end{equation}
Since the domain of $f$ is $(0,1]^{4n}$, $z$ as an argument of $f$ satisfies $z \gg 0$. 
Further, by Lemma \ref{lem:irreducible}, since $Q^*(z)^{-1}M^*(z)$ is an irreducible, non-negative matrix for all $z \gg \edit{\mathbf{0}}$, by Lemma \ref{lem:stricly_pos_vec_M}, there exists a $v \gg \edit{\mathbf{0}}$ such that
\begin{equation}
    Q^*(z)^{-1}M^*(z)v= cv,
\end{equation}
where $c = \rho\big(Q^*(z)^{-1}M^*(z)\big)$. Since $\inf_{t \geq t_0}s(-Q(t) + M(t)) > 0$ we have, by Proposition \ref{prop:spec_rad}, that $c > 1$. We can then find an $\varepsilon > 0$ such that for each $i \in [4n]$
\begin{align}
\label{eq:eps_v_leq_r}
    \varepsilon v_i \leq \frac{c-1}{c}
\end{align}
and
\begin{equation} \label{eq:esp_vi_leq_zi}
    \varepsilon v_i \leq z_i.
\end{equation}
From \eqref{eq:eps_v_leq_r}, it follows that $1\leq \frac{c}{1+\varepsilon c v_i}$ and thus $\varepsilon v_i \leq \frac{\varepsilon c v_i}{1+\varepsilon c v_i}$, yielding 
\begin{align}
\label{eq:eps_v_leq_spec_rad}
    \varepsilon v_i &\leq \frac{\big(Q^*(z)^{-1}M^*(z)\varepsilon v \big)_i}{1+\big(Q^*(z)^{-1}M^*(z)\varepsilon v \big)_i}.
\end{align}
We can show that the right hand side of \eqref{eq:eps_v_leq_spec_rad} is less than or equal to $f(\varepsilon v)_i=\big(Q^*(\varepsilon v)^{-1}M^*(\varepsilon v)\varepsilon v \big)_i$ for each $i \in [4n]$ as follows. We can expand \eqref{eq:equilibrium_eq} using \eqref{eq:full_state_vec} as
\begin{subnumcases}{f(z)= \label{eq:subeq}}
   (diag(B x^z)+\Gamma)^{-1}(Ar^z + \Phi s^z) \label{eq:subeq_s} \\ 
   (\Sigma+\Gamma)^{-1}(diag(Bx^z)s^z + \Phi e^z) \label{eq:subeq_e} \\
   (D+\Gamma)^{-1}(\Sigma e^z + \Phi x^z) \label{eq:subeq_x} \\
   (A + \Gamma)^{-1}(D x^z + \Phi r^z) \label{eq:subeq_r}
\end{subnumcases}
where we can write $z = \begin{bmatrix} {s^z}^\top & {e^z}^\top & {x^z}^\top & {r^z}^\top \end{bmatrix}^\top$. Note that for \eqref{eq:subeq_e}-\eqref{eq:subeq_r}, if we increase any value $z_i$ the output also increases, i.e., \eqref{eq:subeq_e}-\eqref{eq:subeq_r} are monotonic. Thus, since $\varepsilon v_i \leq z_i$, for any $f(\varepsilon v)_i$ in \eqref{eq:subeq_e}-\eqref{eq:subeq_r}, 
we have $\varepsilon v_i \leq f(\varepsilon v)_i$. 

We now consider \eqref{eq:subeq_s}, which is not monotonic, and show that $\varepsilon v_i \leq f(\varepsilon v)_i$,
for the last case \eqref{eq:subeq_s}.
To show this is true, we write the vector $v$ as
\begin{equation*}
    v = \begin{bmatrix} {v^s}^\top & {v^e}^\top & {v^x}^\top & {v^r}^\top \end{bmatrix}^\top
\end{equation*}
and note that by \eqref{eq:esp_vi_leq_zi} we have
\begin{align}
    \beta_i \varepsilon v_i^x + \gamma_i &\leq  \beta_i x_i+\gamma_i,
\end{align}
for every $i \in [n]$. To simplify notation, let 
\begin{equation}
\label{eq:simplified_variable}
    a_i = (\alpha_i v^r_i + (\Phi v^s)_i).
\end{equation}
We then have that
\begin{align}
    \beta_i \varepsilon v_i^x + \gamma_i &\leq  (\beta_i x_i+\gamma_i)(1+ (\beta_i x_i+\gamma_i)^{-1}\varepsilon a_i)
\end{align}
since $(\beta_i x_i+\gamma_i)^{-1}\varepsilon a_i \geq 0$, which then yields
\begin{align}
\label{eq:second_last_ineq}
    \frac{(\beta_i x_i+\gamma_i)^{-1} }{1+ (\beta_i x_i+\gamma_i)^{-1}\varepsilon a_i} &\leq (\beta_i \varepsilon v^x_i +\gamma_i)^{-1}.
\end{align}
Finally, multiplying both sides of \eqref{eq:second_last_ineq} by $\varepsilon a_i$ we have
\begin{align}
\label{eq:final_ineq}
    \frac{(\beta_i x_i+\gamma_i)^{-1} \varepsilon a_i}{1+ (\beta_i x_i+\gamma_i)^{-1}\varepsilon a_i} &\leq (\beta_i \varepsilon v^x_i +\gamma_i)^{-1}\varepsilon a_i.
\end{align}
Therefore, since by \eqref{eq:eps_v_leq_spec_rad}  we have
\begin{align*}
    \varepsilon v_i^s \leq \frac{(\beta_i x_i+\gamma_i)^{-1} \varepsilon (\alpha_i v^r_i + (\Phi v^s)_i)}{1+ (\beta_i x_i+\gamma_i)^{-1}\varepsilon (\alpha_i v^r_i + (\Phi v^s)_i)}
\end{align*}
and by \eqref{eq:subeq_s}
\begin{align*}
    f(\varepsilon v^s)_i = (\beta_i \varepsilon v^x_i +\gamma_i)^{-1}\varepsilon (\alpha_i v^r_i + (\Phi v^s)_i),  
\end{align*}
we have that $\varepsilon v_i^s \leq f(\varepsilon v^s)_i$, for all $i \in [n]$, by \eqref{eq:final_ineq}. Thus, we have shown that $\varepsilon v_i \leq f(\varepsilon v)_i$, for all $i \in [4n]$. 

We now show that $f(z)$ is bounded by a $b \in \R$. Since $\edit{\mathbf{0}} \ll z \leq \edit{\mathbf{1}}$ we have by Assumption \ref{assume:parameters} and \eqref{eq:subeq} that $f(z)$ must also be finite and we can choose $b = \sup_{i \in [4n]} f(z)_i < \infty$. Thus, $f$ maps the convex compact set $\mathcal{C} = \{ z \,\,|\,\, \varepsilon v \leq z \leq  \mathbf{1} b\}$ to itself. By Brouwer’s fixed-point theorem, $f$ has a fixed point in $\mathcal{C}$, which must be strictly positive. Let $z^*$ be this fixed point, then $f(z^*) = z^*$, i.e.
\begin{equation}
     z^* = Q^*(z^*)^{-1}M^*(z^*)z^*.
\end{equation}
Therefore, we have
\begin{align*}
    Q^*(z^*)z^* &= M^*(z^*)z^* \\
    \edit{\mathbf{0}} & = (M^*(z^*) - Q^*(z^*))z^*
\end{align*}
and thus $z^* \gg \edit{\mathbf{0}}$ is an equilibrium point of \eqref{eq:flows_ind_node_cont}. Further, by Lemma \ref{lem:welldefined} we know that $z^* \in (0,1]^{4n}$.
\end{proof}
\noindent 
Note that since our method of proving existence of an endemic equilibrium for \eqref{eq:flows_ind_node_cont} requires irreducibility we must include the susceptible state in our analysis, causing the sufficient condition needed to prove existence to be stronger than one that is time invariant. We now give conditions under which the endemic equilibrium
is unique.
\begin{theorem} \label{thm:endemic_unique}
Let there exist an endemic equilibrium for the system in \eqref{eq:flows_ind_node_cont}. If $\beta_i \geq \gamma_i$, for all $i \in [n]$, then the endemic equilibrium is unique.
\end{theorem}
\begin{proof}
Let $z$ and $y$ be two nonzero equilibria of \eqref{eq:flows_ind_node_cont} which can be divided into their respective epidemic states as $z = \begin{bmatrix} {s^z}^\top & {e^z}^\top & {x^z}^\top & {r^z}^\top \end{bmatrix}^\top$ and $y = \begin{bmatrix} {s^y}^\top & {e^y}^\top & {x^y}^\top & {r^y}^\top \end{bmatrix}^\top$. We can express $z$ in terms of $y$ for $i \in [n]$ as
\begin{equation}
\label{eq:pert_def}
\begin{aligned}
    (s_i^y+ \varepsilon_{s_i}) + (e_i^y+ \varepsilon_{e_i}) + (x_i^y+ \varepsilon_{x_i}) + (r_i^y+ \varepsilon_{r_i})  \\= s_i^z + e_i^z + x_i^z + r_i^z = 1,
\end{aligned}    
\end{equation}
where
\begin{align} \label{eq:sum_pert_0}
    \varepsilon_{s_i} + \varepsilon_{e_i} + \varepsilon_{x_i} + \varepsilon_{r_i} = 0,
\end{align}
for each $i \in [n]$. Since both $z \gg 0$ and $y \gg 0$ are fixed points of \eqref{eq:flows_ind_node_cont}, we can sum \eqref{eq:flows_ind_node_cont_e}, \eqref{eq:flows_ind_node_cont_x}, and \eqref{eq:flows_ind_node_cont_r} to obtain
\begin{align*}
    \dot{e}_i^z+\dot{x}_i^z+\dot{r}_i^z &=
    \beta_i x_i^z \edit{s_i^z} -\gamma_i(e_i^z+x_i^z+r_i^z)  
    - \alpha_i r_i^z + \sum_{j \neq i} \phi_{ij}(e_j^z+x_j^z+r_j^z) \\
    &= \beta_i (x_i^y + \varepsilon_{x_i}) (s_i^y+ \varepsilon_{s_i}) - \alpha_i (r_i^y+ \varepsilon_{r_i}) -\gamma_i(e_i^y+x_i^y+r_i^y+\varepsilon_{e_i}+\varepsilon_{x_i}+\varepsilon_{r_i}) \\ 
    & \,\,\,\,\,\, + \sum_{j \neq i} \phi_{ij}(e_j^y+x_j^y+r_j^y+\varepsilon_{e_j}+\varepsilon_{x_j}+\varepsilon_{r_j}) \\
    &= \beta_i x_i^y s_i^y -\gamma_i(e_i^y+x_i^y+r_i^y)
    - \alpha_i r_i^y + \sum_{j \neq i} \phi_{ij}(e_j^y+x_j^y+r_j^y) \\
    &= \dot{e}_i^y+\dot{x}_i^y+\dot{r}_i^y = 0,
\end{align*}
where $\phi_{ij} = \frac{N_j}{N_i} w_{ij}\gamma_j$. Thus, by taking the difference of $(\dot{e}_i^z+\dot{x}_i^z+\dot{r}_i^z) - (\dot{e}_i^y+\dot{x}_i^y+\dot{r}_i^y)$ we get
\begin{equation} \label{eq:esp_condit_eq0}
\begin{aligned}
    \beta_i \big(x_i^y \varepsilon_{s_i} + s_i^y \varepsilon_{x_i} + \varepsilon_{x_i}\varepsilon_{s_i}  \big)  -\gamma_i(\varepsilon_{e_i}+\varepsilon_{x_i}+\varepsilon_{r_i})
    - \alpha_i \varepsilon_{r_i}+ \sum_{j \neq i} \phi_{ij}(\varepsilon_{e_j}+\varepsilon_{x_j}+\varepsilon_{r_j}) &=0.
\end{aligned}
\end{equation}
We can express \eqref{eq:esp_condit_eq0} for all $i\in [n]$ using matrix notation as follows. Define $\varepsilon_q = [\varepsilon_{q_1}, \dots, \varepsilon_{q_n}]^\top$, where $q \in \{s,e,x,r\}$. Then, using \eqref{eq:sum_pert_0} and \eqref{eq:esp_condit_eq0} we have
\begin{equation} \label{eq:vec_esp_pert}
    B(X^y \varepsilon_s+ S^y \varepsilon_x + E_x \varepsilon_s) + \Gamma \varepsilon_s - A \varepsilon_r - \Phi \varepsilon_s = \mathbf{0},
\end{equation}
\noindent where $X^y = diag(x_i^y)$, $S^y = diag(s_i^y)$, and  $E_x = diag(\varepsilon_{x_i})$. Further, we can put \eqref{eq:vec_esp_pert} into a block matrix form as
\begin{align}
\label{eq:mat_esp_pert}
\underbrace{
 \begin{bmatrix}
 BX^y + E_x + \Gamma - \Phi & 0 & 0 \\
 0 & BS^y & 0 \\
 0 & 0 & -A
 \end{bmatrix}
 }_{K}
 \begin{bmatrix}
  \varepsilon_s \\ \varepsilon_x \\ \varepsilon_r
 \end{bmatrix}
 = \mathbf{0}.
 \end{align}
 Thus, the solution to the vector of perturbations to $s^y, x^y, r^y$ is the nullspace of $K$. Note that if $K$ is full rank, then the only solution for the perturbations must be $\varepsilon_{s}=\varepsilon_{x}=\varepsilon_{r} = \mathbf{0}$.
Considering the rank of $K$, it is 
clear that the second and third block diagonals are full rank since $y \gg 0$ and by Assumption~\ref{assume:parameters}, respectively. Further, by Lemma \ref{lem:diag_dominant}, we have that if $BX^y + E_x + \Gamma - \Phi$ is diagonally dominant then it is also full rank, or in other words, if
 \begin{equation}
 \label{eq:diag_dom}
     |\beta_i x_i^y + \varepsilon_{x_i} +\gamma_i| > \sum_{j \neq i} |\phi_{ij}|, \,\,\, \forall i \in [n],
 \end{equation}
 then $\det(BX^y + E_x + \Gamma - \Phi) \neq 0$. 
 
We will now show that \eqref{eq:diag_dom} holds for any choice of $\varepsilon_{x_i}$ if $\beta_i \geq 1$ for all $i \in [n]$. Note that when $\varepsilon_{x_i} \geq 0$ we have that \eqref{eq:diag_dom} is satisfied, since $\gamma_i = \sum_{j \neq i} \phi_{ij}$ for all $i \in [n]$ by Assumption~\ref{assume:equal_flows}, and $\beta_i x^y_i > 0$ by Assumption~\ref{assume:parameters}. We now consider when $\varepsilon_{x_i} < 0$. Since $x_i^y + \varepsilon_{x_i} = x_i^z > 0$, we have that $x_i^y > -\varepsilon_{x_i}$. Therefore, if $\varepsilon_{x_i} < 0$, then we have that $x_i^y > |\varepsilon_{x_i}|$. Thus, if $\beta_i \geq 1$ for all $i \in [n]$, then \eqref{eq:diag_dom} must hold for any feasible choice of $\varepsilon_{x_i}$. Therefore, by \eqref{eq:mat_esp_pert} we have $\varepsilon_{s}=\varepsilon_{x}=\varepsilon_{r} = \mathbf{0}$, which implies, by \eqref{eq:sum_pert_0}, that $\varepsilon_{e} = \mathbf{0}$. Thus, by \eqref{eq:pert_def}, $z=y$.
 
 Now consider a system with an endemic equilibrium where $\exists i \in [n]$ such that $\beta_i < 1$. If $\beta_i \geq \gamma_i$ for all $i \in [n]$, there exists some constant scalar $\eta \in \mathbb{R}$ such that by \eqref{eq:matrix_eq},
 \begin{align*}
     \eta \left(A r^* -(B X^*+\Gamma) s^* +N^{-1}W \Gamma N s^* \right) = \edit{\mathbf{0}}  \\
     \eta \left(B X^* s^* - (\Sigma+\Gamma)e^* + N^{-1}W \Gamma N e^* \right) = \edit{\mathbf{0}},
\end{align*}
where $\eta \beta_i \geq 1$ and $\eta \gamma_i \leq 1$ for all $i \in [n]$. Thus, any endemic system with parameters where $\exists i \in [n]$ such that $\beta_i < 1$ can be mapped to one with the same endemic equilibrium where $\beta_i \geq 1$ for all $i \in [n]$.
\end{proof}
\begin{remark}
It should be noted that the condition $\beta_i \geq \gamma_i, \forall i \in [n],$ for uniqueness of the endemic equilibrium is only sufficient and simulations show that an endemic equilibrium can still exist even when $\exists i \in [n]$ such that $\beta_i < \gamma_i$.
\end{remark}
The analysis throughout this section establishes several key points. First, we find that short of a complete lockdown of infected nodes before spreading begins, manipulating travel flows alone cannot change the long-term stability of the healthy state of this system. Further, we have established conditions under which the healthy state is exponentially stable, the healthy state is unstable, an endemic equilibrium
exists, and the endemic equilibrium is unique.

\section{Applications of Travel Flows Model}
\label{sec:applications}
 In this section, we shift our focus to constructing novel tools that will enable us to leverage real-world data in Section~\ref{sec:case_study}. The proposed tools are 1) learning the infection spread parameters given infection states and travel flows data and 2) using travel flows data to predict the arrival time of a disease to every node in the network given a disease origin. In this section, we illustrate these tools using simulated data. \edit{We share the simulation code used for this section on GitHub, which can be found in \cite{butlercode2023}.}

\subsection{Parameter Identification} \label{sec:paramID}
In this section we present a method for estimating the infection parameters $\beta_i , \sigma_i , \delta_i , \alpha_i$ for each node $i$, respectively, given measurement data. Since infection data is generally collected at discrete time intervals, we first discretize the system in \edit{\eqref{eq:flows_ind_node_cont}} using Euler's method which yields

\begin{subequations}
\small
\label{eq:seirs_flows_desc}
\begin{align}
    s_i^{k+1} &= s_i^k + h\left(\alpha_i r_i^k -(\beta_i x_i^k+ \gamma_i^k) s_i^k +\sum_{j \neq i} \frac{N_j}{N_i} w^k_{ij}\gamma^k_j s^k_j \right)\\
    e_i^{k+1} &=e_i^k+ h\left(\beta_i x_i^k s_i^k - (\sigma_i + \gamma_i^k )e_i^k+ \sum_{j \neq i} \frac{N_j}{N_i} w^k_{ij}\gamma^k_j e^k_j\right) \\
    x_i^{k+1} &= x_i^k+h\left(\sigma_i e_i^k - (\delta_i+\gamma_i^k) x_i^k+ \sum_{j \neq i} \frac{N_j}{N_i} w^k_{ij}\gamma^k_j x^k_j\right) \\        
    r_i^{k+1} &= r_i^k+h\left(\delta_i x_i^k- (\alpha_i + \gamma_i^k) r_i^k + \sum_{j \neq i} \frac{N_j}{N_i} w^k_{ij}\gamma^k_j r^k_j\right), 
\end{align}
\end{subequations}
where $h>0$ is a sampling parameter and $k \in \mathbb{Z}_{\geq 0}$ is a given time index. \edit{For this application, we assume a sampling parameter $h$ that is small enough such that the system remains well-defined given the relative values of the model parameters. For a more detailed discussion and analysis of exactly when this condition holds, see our previous work in \cite{butler2021effect}.} Given $T$ samples of infection state data, we can estimate the parameters for node $i$ as follows. Let
\begin{align}
\footnotesize
    \Delta q_i =
    \begin{bmatrix}
     q^1_i - q^0_i +h(\gamma^0_i q^0_i - \sum_{j \neq i} \frac{N_j}{N_i} w^0_{ij}\gamma^0_j q^0_j ) \\ \vdots \\ 
     q^{T}_i - q^{T-1}_i +h(\gamma^{T-1}_i q^{T-1}_i - \sum_{j \neq i} \frac{N_j}{N_i} w^{T-1}_{ij}\gamma^{T-1}_j q^{T-1}_j )
    \end{bmatrix}
\end{align}
where $q_i \in \{s_i,e_i,x_i,r_i \}$, and define
\begin{align}
\footnotesize
\label{eq:ID_matrix}
    \Psi_i = h
    \begin{bmatrix}
     -s_i^0 x_i^0 & 0 & 0 & r_i^0 \\
     \vdots & \vdots & \vdots & \vdots \\
     -s_i^{T-1} x_i^{T-1} & 0 & 0 & r_i^{T-1} \\
     s_i^0 x_i^0 & -e_i^0 & 0 & 0 \\
     \vdots & \vdots & \vdots & \vdots \\
     s_i^{T-1} x_i^{T-1} & -e_i^{T-1} & 0 & 0 \\
     0 & e_i^0 & -x_i^0 & 0 \\
     \vdots & \vdots & \vdots & \vdots \\
     0 & e_i^{T-1} & -x_i^{T-1} & 0 \\
     0 & 0 & x_i^0 & -r_i^0 \\
     \vdots & \vdots & \vdots & \vdots \\
     0 & 0 & x_i^{T-1} & -r_i^{T-1} \\
    \end{bmatrix},
\end{align}
then, we rewrite the system in \eqref{eq:seirs_flows_desc} for a node $i$ given $T$ data samples of data in terms of the spread parameters as
\begin{align}
    \begin{bmatrix}
     \Delta s_i \\ \Delta e_i \\ \Delta x_i \\ \Delta r_i 
    \end{bmatrix}
    =
    \Psi_i
    \begin{bmatrix}
     \beta_i \\ \sigma_i \\ \delta_i \\ \alpha_i
    \end{bmatrix}.
\end{align}
Thus, we can solve for the parameters $\beta_i , \sigma_i , \delta_i , \alpha_i$ as 
\begin{align}
\label{eq:param_sol}
    \begin{bmatrix}
     \beta_i \\ \sigma_i \\ \delta_i \\ \alpha_i
    \end{bmatrix}
    =
    \Psi_i^{\dagger}
    \begin{bmatrix}
     \Delta s_i \\ \Delta e_i \\ \Delta x_i \\ \Delta r_i 
    \end{bmatrix},
\end{align}
where $\Psi_i^{\dagger}$ is the pseudo-inverse of $\Psi_i$. Thus, we have that a solution to the spread parameters using \eqref{eq:param_sol} is unique if and only if \eqref{eq:ID_matrix} is full column rank. 

Using this method of parameter estimation, we can utilize real infection data along with real travel flows data to determine the spreading rates for a given disease outbreak, assuming that the model defined in Section \ref{sec:model} approximates the behavior of the outbreak. We show an example of parameter identification in Figure~\ref{fig:param_fit_sim}, where data from a \edit{5-node} simulated system \edit{with initial conditions $s^0 = [0.549,0.715,0.603,0.545,0.424]^\top$, $e^0 = \mathbf{1}-s^0$, and $x^0 = r^0 = \mathbf{0}$} is perturbed with Gaussian noise \edit{with zero mean and standard deviation 0.01} and then used to learn spread the infection spread parameters as outlined in this section\edit{, where the travel weight matrix used is given in Table~\ref{tab:trav_weights} and the spread parameters and RMSE of the learned parameters are given in Table~\ref{tab:trav_parameters_RMSE}.}

\begin{table}[]
\centering
\begin{tabular}{c|ccccc}
$w_{ij}$ & 1     & 2     & 3     & 4     & 5     \\ \hline
1    & 0     & 0.212 & 0.275 & 0.25  & 0.212 \\
2    & 0.249 & 0     & 0.26  & 0.299 & 0.338 \\
3    & 0.246 & 0.198 & 0     & 0.204 & 0.178 \\
4    & 0.285 & 0.29  & 0.259 & 0     & 0.272 \\
5    & 0.22  & 0.299 & 0.206 & 0.247 & 0    
\end{tabular}
\caption{\edit{Travel connection matrix used in the simulated system shown in Figure~\ref{fig:param_fit_sim}.}}
\label{tab:trav_weights}
\end{table}

\begin{table}[]
\centering
\begin{tabular}{c|ccccc|c|}
Node  & 1     & 2     & 3     & 4     & 5     & RMSE  \\ \hline
$\beta_i$  & 0.065 & 0.044 & 0.089 & 0.096 & 0.038 & 0.03  \\
$\sigma_i$ & 0.079 & 0.053 & 0.057 & 0.093 & 0.007 & 0.01  \\
$\delta_i$ & 0.001 & 0.001 & 0.008 & 0.008 & 0.009 & 0.002 \\
$\alpha_i$ & 0.01  & 0.008 & 0.005 & 0.008 & 0.001 & 0.012 \\
$\gamma_i$ & 0.002 & 0.002 & 0.002 & 0.002 & 0.005 & - 
\end{tabular}
\caption{\edit{Parameters used in the simulated system shown in Figure~\ref{fig:param_fit_sim} along with the RMSE of learned parameters.}}
\label{tab:trav_parameters_RMSE}
\end{table}

\begin{figure}
    \centering
    \includegraphics[width=.7\columnwidth]{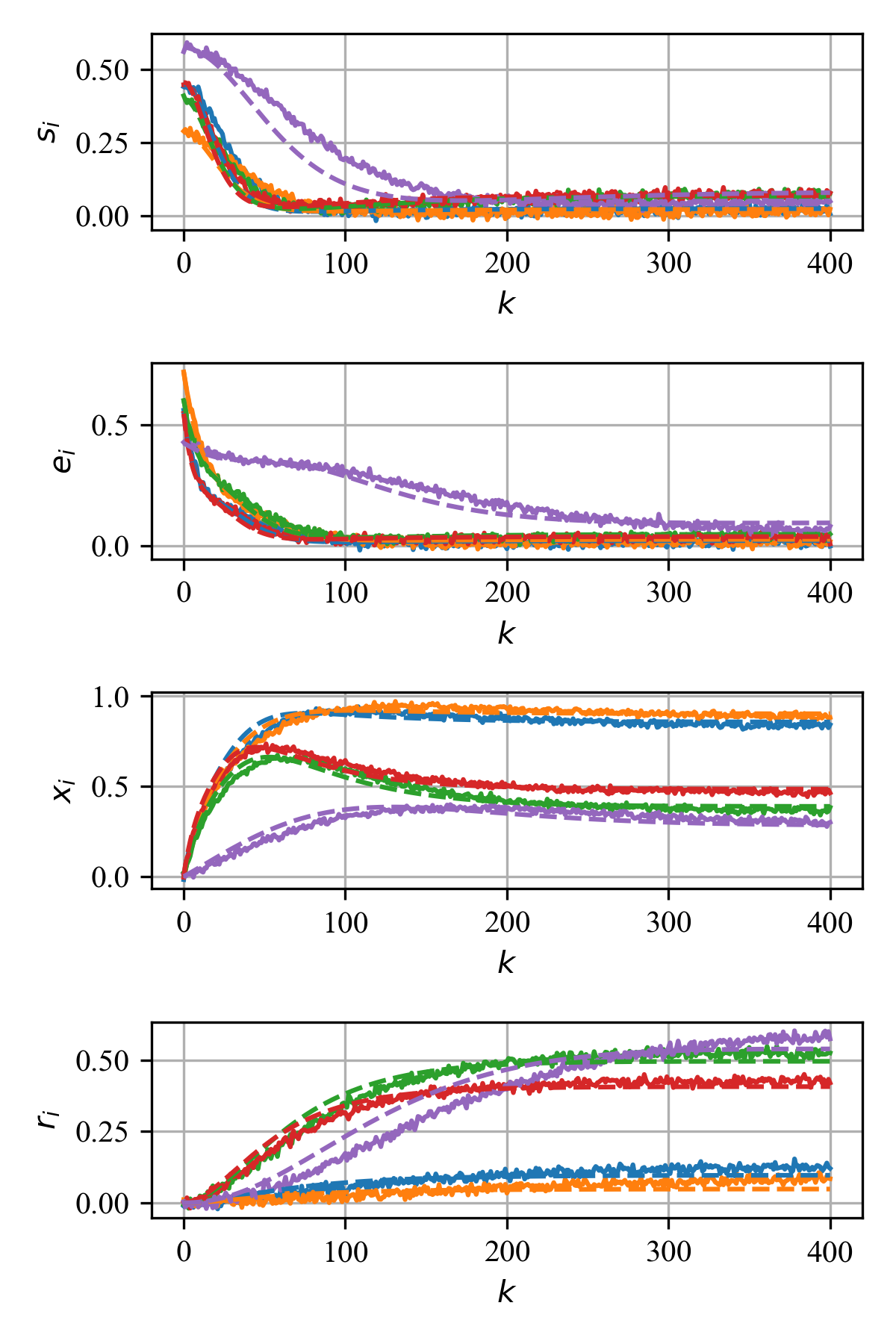}
    \caption{Simulated \edit{5-node} system from \eqref{eq:seirs_flows_desc} \edit{with parameters and travel matrix defined in Tables~\ref{tab:trav_parameters_RMSE} and \ref{tab:trav_weights}} where $n=5$, $h=1$, and added mean zero Gaussian noise with standard deviation \edit{$0.01$} and spread parameters chosen to induce an endemic state. \edit{Each solid line represents the perturb simulated data for each node.} Parameters from the noisy system are identified and then re-simulated from the same initial condition, \edit{where the predicted state trajectories for each node are shown by dashed lines of the same color.}}
    \label{fig:param_fit_sim}
\end{figure}

\subsection{Effective Distance and Travel Flows} \label{sec:eff_dist_flows}
Given our notion of population flow and the definition of $w_{ij}$ being the proportion of traveling members of the population at node $j$ flowing to node $i$, we can use this proportional flow information to compute the effective distance between any two nodes based on the most probable path an individual would take between nodes. Further, we can then use the effective distance between nodes to make predictions of the arrival time of an outbreak to a given node based on the distance of that node from other infected nodes. A similar notion of effective distance based on travel and using it to predict disease arrival time is introduced and used in \cite{brockmann2013hidden}, where the effective distance between countries is computed using flight traffic data. In our case, we apply this measure to inter-county travel data in Minnesota during a period of the COVID-19 pandemic, which will be shown in Section~\ref{sec:case_study}. 

We compute the effective distance between nodes as follows. We interpret the mobility matrix defined by the $w_{ij}$ as the transition matrix of a discrete time homogeneous stochastic jump process, where at each time step a randomly moving particle moves from node $j$ to node $i$ with probability $w_{ij}$. Let $n_k$ be the location of this particle at time $k$. Now, consider the path of a randomly moving particle after $L$ steps, forming the path $\Xi = \{n_0, n_1, \dots, n_L \}$. If we fix any two endpoints in the graph where $\Xi_{ij}$ is any path with $n_0=j$ and $n_L=i$, we can associate any path taken between these points with a probability
\begin{align}
    F(\Xi_{ij}) = \edit{w_{i n_{L-1}} \times \dots \times w_{n_1 j}} = \prod_{k=1}^{L} w_{n_k n_{k-1}}.
\end{align}
Since the proposed notion of distance is additive, we establish a connection between the probability of a given path and the effective distance of a given path by employing the multiplicative property of logarithms
\begin{equation}
    D_{F(\Xi_{ij})} = - \log F(\Xi_{ij}) = -\sum_{k=1}^{L}\log w_{n_k n_{k-1}}, 
\end{equation}
\noindent where $D_{F(\Xi_{ij})}$ is the log sum of the probabilities of each step taken on the given path.
Thus, since $w_{ij} \in [0,1]$ a low probability path will equate to a high effective distance, with $D_{F(\Xi)} \rightarrow \infty$ as $F(\Xi) \rightarrow 0$, and high probability paths result in lower effective distance. We construct a new graph with edge weights between nodes being the one-step negative log probabilities of traveling between nodes as
\begin{equation}
\label{eq:dist_graph}
    d_{ij} =
    \begin{cases}
    0 & i=j \\
    - \log w_{ij} & w_{ij} > 0 \\
    \infty & w_{ij} = 0.
    \end{cases}
\end{equation}
Using this graph with edge weights between nodes $i$ and $j$ defined by \eqref{eq:dist_graph}, we compute the effective distance as 
\begin{equation}
\label{eq:eff_dist_origin}
    D_{ij} = \min_{\Xi_{ij}} D_{F(\Xi_{ij})} \geq 0,
\end{equation}
where $D_{ij}$ is the sum of the path lengths between nodes $j$ and $i$ that yield the smallest distance, which can be found for any node $i \in [n]$ by computing the minimum spanning tree for the distance graph from node $j$\edit{, since} $D_{F(\Xi_{ij})} = \sum_{k=1}^{L} d_{n_k n_{k-1}}$ describes the cost of traveling any path between $i$ and $j$ in the graph.
It is important to note that since it is possible $d_{ij} \neq d_{ji}$, we also have that $D_{ij} \neq D_{ji}$. Further, the shortest path to and from any two endpoints may not necessarily contain the same intermediate points.

We can extend the notion of the most probable path and effective distance further by considering the most probable path from a subgroup of nodes to any given node. Since we are interested in the effective distance for an infection to travel to uninfected nodes with respect to an infected group of nodes, we can construct a node subgroup as follows
\begin{equation}
    \mathcal{X} = \big\{ i \in [n]; x_i > p \big\},
\end{equation}
where $x_i$ is the infection level at node $i$ and $p \in (0,1]$ is some threshold at which we consider the node to be infected. We can then compute the probability that an individual from node group $\mathcal{X}$ travels to node $i$ as
\begin{equation}
    \tilde{w}_{i\mathcal{X}} =\frac{1}{\sum_{l \in \mathcal{X}} N_l} \sum_{j \in \mathcal{X}} N_j w_{ij}.
\end{equation}
Thus, we can construct another effective distance graph, with respect to a group of nodes $\mathcal{X}$, using edge weights
\begin{equation}
\label{eq:dist_graph_window}
    \tilde{d}_{ij}^{\mathcal{X}} =
    \begin{cases}
    0 &  i \in \mathcal{X} \text{ or } i=j \\
    - \log \tilde{w}_{i\mathcal{X}} & i \notin \mathcal{X}, j \in \mathcal{X}, \tilde{w}_{i\mathcal{X}} > 0 \\
    - \log w_{ij} & i,j \notin \mathcal{X},  w_{ij} > 0 \\
    \infty & \text{otherwise},
    \end{cases}
\end{equation}
\noindent making the effective distance in the case from any node $j$ to every other node $i \in [n]$, with respect to the node group~${\mathcal{X}}$,
\begin{equation}
\label{eq:eff_dist_Xcal}
    \tilde{D}_{ij}^{\mathcal{X}} =\min_{\Xi_{ij}} \big( \tilde{d}_{j n_{L-1}}^{\mathcal{X}} + \dots + \tilde{d}_{n_1 i}^{\mathcal{X}} \big).
\end{equation}

When a disease outbreak occurs, one important question for policymakers is how long it will take for the disease to infect new nodes given past infection data. Given travel data and knowledge of which nodes are currently infected, we can use \eqref{eq:eff_dist_Xcal} to predict the arrival time of the disease to uninfected nodes as follows. Let $T_k$ be the time of the $k$th arrival of the disease to an uninfected node,  and suppose $k > \tau$ disease arrivals have occurred, where $\tau$ is the size of your training window:
\begin{enumerate}
    \item At time $T_k$, select training data that includes the disease arrivals from time $T_{k-\tau}$ to $T_{k}$, where $\tau = |\{T_{k-\tau}, T_{k-\tau+1}, \dots, T_k \}|$
    \item Compute the effective distance of each node in the training data to the infected group $\mathcal{X}$ at time $T_{k-\tau}$
    \item Given the arrival times of the training data and the computed effective distances to the infected group, fit a line to these training data 
    \item Using the fitted line, predict the arrival time of the remaining uninfected nodes given their effective distance to the current infected group $\mathcal{X}$ at time $T_k$.
\end{enumerate}
\noindent To ensure that this method always predicts sensible arrival times, as the fitted line may predict arrival times that have already passed, a constant shift is added to the predictions such that the next predicted arrival time is greater than zero. This method differs from \cite{brockmann2013hidden} most significantly in its ability to make updated predictions on the arrival time of an outbreak, which depends on the current flow rate. Further, it facilitates an updated notion on the distance of an outbreak from any given node based on the likelihood of transmission from multiple source nodes as the epidemic progresses, described by \eqref{eq:dist_graph_window}, rather than relying on the distance from the original source. These factors contribute to a more accurate prediction of arrival time in a more realistic evolving scenario. 

We illustrate this algorithm by simulating the spread of a disease over a network using our model as defined in Section~\ref{sec:model}. We construct our network to emulate counties in Minnesota with travel between counties given by inter-county travel data collected during part of the COVID-19 pandemic. A full description of this data is given in Section~\ref{sec:case_study}. The system is then simulated using infection parameters satisfying Assumption~\ref{assume:parameters} and sampling parameter $h=1$ (to emulate a daily resolution of data collection), where a node is selected as the origin of the disease and predictions on the next nodes to be infected are made. Using \eqref{eq:eff_dist_origin}, we compute the effective distance of each node from the origin of infection and plot the effective distance versus the arrival time of the disease to each node, shown in Figure~\ref{fig:sim_mn_arrival_vs_effdist}. Note the linear relationship between effective distance and arrival time.
\begin{figure}
    \centering
    \includegraphics[width=.6\columnwidth]{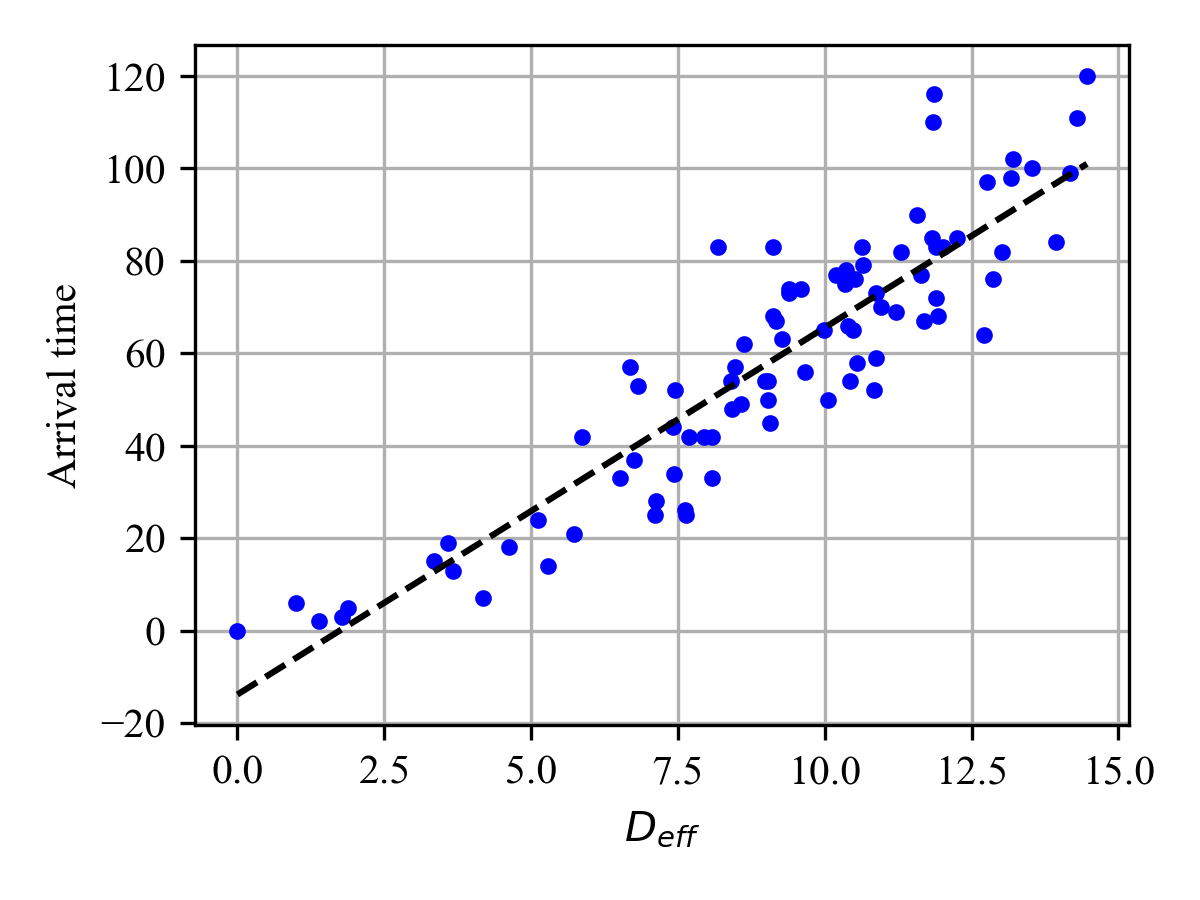}
    \caption{The arrival time of the disease to each node versus the effective distance of each node in our simulated system from the disease origin, computed by \eqref{eq:eff_dist_origin}. The dashed line shows the linear fit of all data points.}
    \label{fig:sim_mn_arrival_vs_effdist}
\end{figure}

In Figure~\ref{fig:sim_k_day_predict} we show an example of using the shifting window algorithm to predict the remaining infection arrival times with an arrival window of $\tau = 20$ disease arrival events. To evaluate the prediction error of our shifting window method, we compute the root-mean-square (RMS) error of the linear fit for all arrival times versus effective distance from the origin at the start of the epidemic (with the linear fit shown in Figure~\ref{fig:sim_mn_arrival_vs_effdist}) and compare it with the average RMS error of the shifting window method that predicts the next 10 arrival times whenever a new arrival occurs. This comparison yields an RMS error of 4.59 timesteps (analogous to days since $h=1$) for the linear fit of all arrival times versus effective distance from the disease origin (which can only be evaluated after all data has been collected) versus an average RMS error of 1.91 timesteps for the shifting window predicting the next 10 arrival times whenever a new arrival occurs, which is a 58\% reduction in prediction error in addition to being usable in real-time to predict the next immediate disease arrivals versus analysis after the fact. 

\begin{figure}
\centering
\begin{subfigure}{0.45\columnwidth}
    \includegraphics[width=\textwidth]{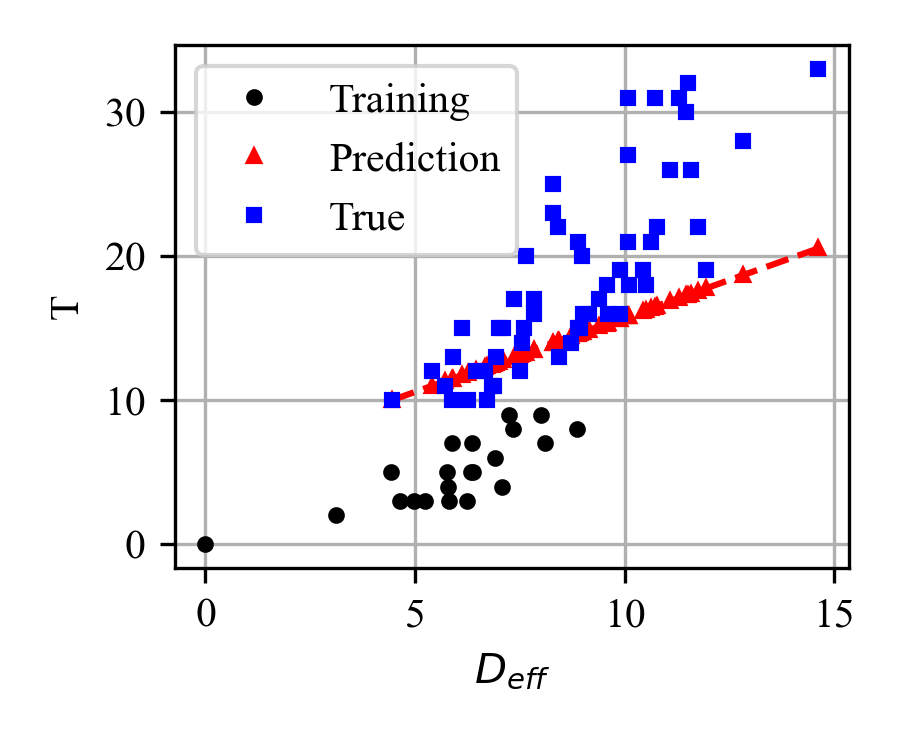}
    \caption{$T=T_{50}$}
\end{subfigure}
\begin{subfigure}{0.45\columnwidth}
    \includegraphics[width=\textwidth]{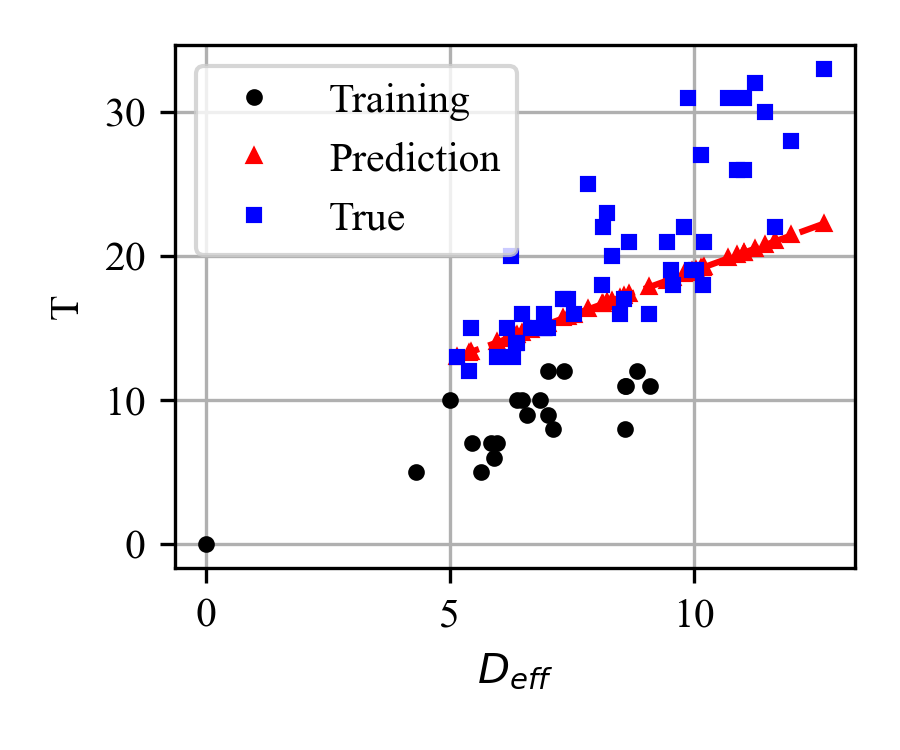}
    \caption{$T=T_{60}$} 
\end{subfigure} 
\begin{subfigure}{0.45\columnwidth} 
    \includegraphics[width=\textwidth]{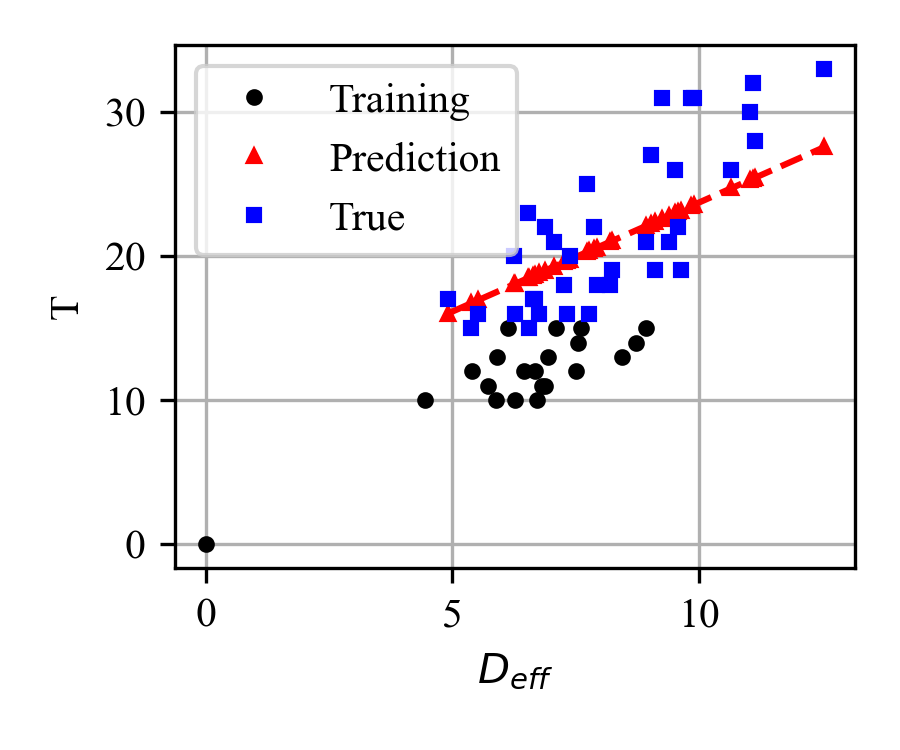} 
    \caption{$T=T_{70}$} 
\end{subfigure}  
\begin{subfigure}{0.45\columnwidth} 
    \includegraphics[width=\textwidth]{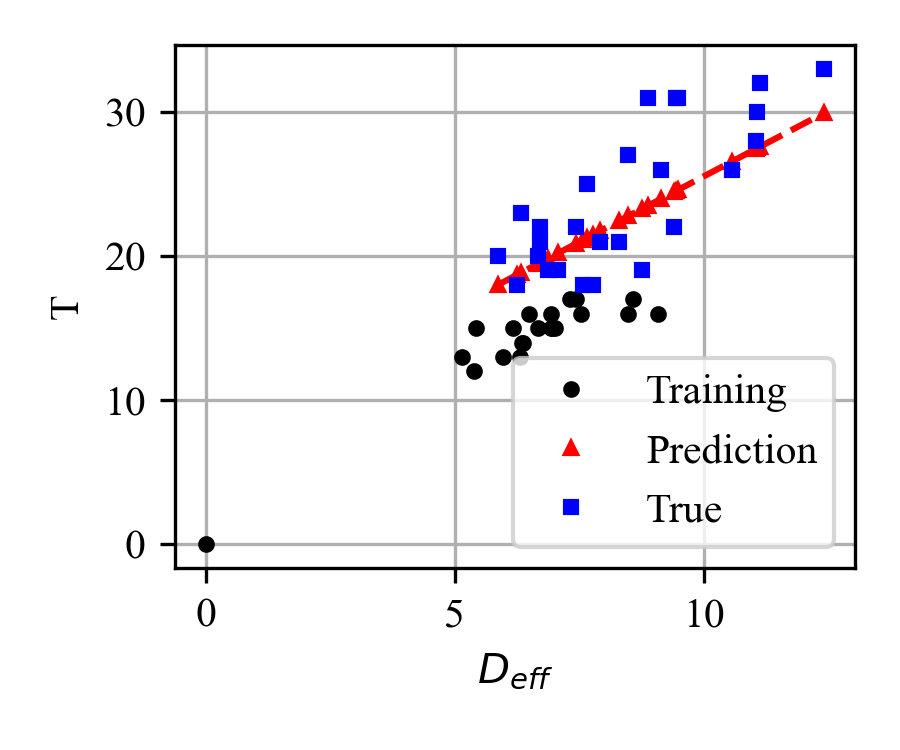} 
    \caption{$T=T_{80}$} 
\end{subfigure}
\caption{Example of shifting prediction window of arrival times in the simulated system at four different arrival times during the spreading process ($T=T_{50},T_{80}$) where training window $\tau=20$, with the data used to make predictions (black circles), true arrival times for each node versus effective distance (blue squares), and shifted best-fit predictions (red triangles) shown at several arrival intervals, where the effective distances are computed at each disease arrival using \eqref{eq:eff_dist_Xcal}.}
\label{fig:sim_k_day_predict}
\end{figure}

\section{Case Study: Minnesota and COVID-19}
\label{sec:case_study}
We apply our model developed in Section \ref{sec:model} and methods described in Section \ref{sec:applications} to data collected from the beginning of the 2020 COVID-19 pandemic in Minnesota. Daily infection data at the county level is taken from usafacts.org which contains the number of confirmed COVID-19 cases over the period of March 2020 to December 2020 \cite{MNdata}. To infer the infection states $s_i,e_i,x_i,r_i$ for each county $i$ from the confirmed cases data we apply a constant shift that assumes confirmed individuals became exposed one week before testing positive, at which point they are considered infected. Then, we assume that infected individuals remain infected for one week after testing positive, after which they are considered recovered. Finally, to account for the rate at which individuals can lose immunity, we transfer recovered individuals back into being susceptible after six weeks of being recovered. 
For travel data, we worked with the Minnesota Department of Transportation to obtain data collected by StreetLight \cite{streetlight} that estimates the number of trips between counties via anonymous geolocalization using smartphones. This provides an estimate of the total number of trips made by individuals between counties over a specified period of time. We choose a weekly time scale in an effort to average out periodic behaviors such as commuting and use this average to estimate the daily flow of individuals between counties.

\subsection{Parameter Identification}
Using the inferred infection states for each county in MN along with the travel flows computed from the collected traffic data, we can use \eqref{eq:param_sol} to solve for the best-fit infection parameters using a convex solver and constraining solutions to be strictly non-negative using our model defined in Section~\ref{sec:model}. We then re-simulate the model using the same initial conditions and learned parameters, with the results of the simulation compared with the inferred infection states shown in Figure~\ref{fig:param_fit_flowsmodel}, \edit{where we show a subset of real infection levels at five counties (solid lines) and the predicted infection levels from the learned parameters (dashed lines)}. Note that although the predicted infection peak of the model with learned parameters is much higher than the actual data, we see that the peak infection time is more closely approximated by the travel flows model, as shown in Table~\ref{tab:param_predict_error}. 

\edit{
Regarding the error in magnitude prediction, there are several likely contributing factors. First, in our estimation algorithm, we assume spread parameters ($\beta_i,\sigma_i,\delta_i,\alpha_i$) that are not time-varying, which is most likely not the case in reality since lock-downs, social distancing policies, medical interventions, and virus mutations can all change the spread parameters over time, which violates this assumption. Additionally, we are assuming that our rather naive method for state estimation (by choosing a priori the recovering rate and the rate at which exposed individuals become infected) yields the correct states for all counties and that all counties are reporting infection cases in the same manner, which may not be true in all cases. Thus, these combinations of structural assumptions and naivety in state estimation are both likely large contributors to the magnitude error in state prediction.
}

For comparison, we apply this parameter fit with another similar model, modified from our previous work in \cite{vrabac2021capturing} to allow individuals to become reinfected. The major difference between our model in Section~\ref{sec:model} and this modified model from \cite{vrabac2021capturing} is that travel data is used to infer direct interaction between sub-populations rather than explicitly modeling the flow of individuals between subpopulations. We then take the best-fit parameters for this model given the infection and travel data and simulate the model with the learned parameters, shown in Figure~\ref{fig:param_fit_LCSSmodel}, \edit{where we show a subset of real infection levels at the same five counties as Figure~\ref{fig:param_fit_flowsmodel} (solid lines) and the predicted infection levels from the learned parameters (dashed lines)}. Note that in this case, although the error is smaller in magnitude than the prediction from the parameter fitting of the travel flows model, the direct interaction model from \cite{vrabac2021capturing} fails to reproduce the peaking behavior shown clearly in the inferred infection states from the COVID-19 data and in Table~\ref{tab:param_predict_error}.

\begin{table}[]
\centering
\begin{tabular}{|c|cc|}
\hline
 & Infection RMSE & Peak Time RMSE \\ \hline
SEIRS (Travel Flows) & 0.023          & 4.15 Weeks     \\
SEIRS (\cite{vrabac2021capturing})    & 0.007          & 21.21 Weeks \\ \hline
\end{tabular}
\caption{The infection RMSE, defined as $\sqrt{\frac{1}{nT}\sum_{i \in [n],k\in[T]}\left( x_i^k - \hat{x}_i^k\right)^2}$, where $n$ is the number of nodes, $T$ is the number of samples, and $\hat{x}_i^k$ is the predicted infection level \edit{at} node $i$ \edit{for} time step $k$, and the peak time RMSE, defined as $\sqrt{\frac{1}{n} \sum_{i \in [n]} \left(\argmax_{k}(x_i^k) - \argmax_{k}(\hat{x}_i^k) \right)^2}$, for the parameter fitting predictions of the SEIRS travel flows model defined in Section~\ref{sec:model} (shown in Figure~\ref{fig:param_fit_flowsmodel}) and modified SEIRS model from \cite{vrabac2021capturing} (shown in Figure~\ref{fig:param_fit_LCSSmodel}), respectively, on the COVID-19 infection data.
}
\label{tab:param_predict_error}
\end{table}

\begin{figure}
    \centering
    \captionsetup{aboveskip=-1pt,belowskip=-12pt}
    \includegraphics[width=.7\columnwidth]{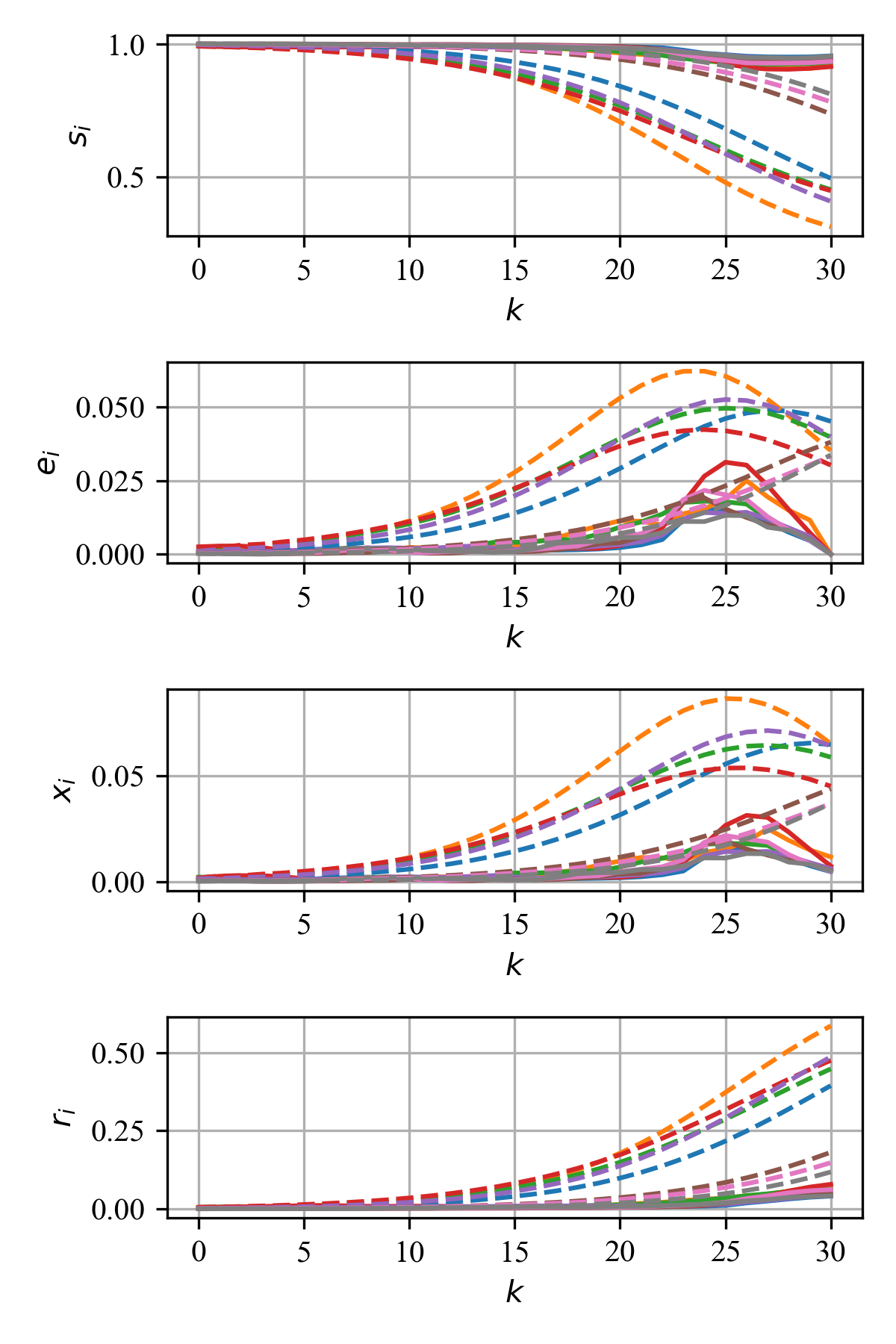}
    \caption{Simulating \eqref{eq:flows_ind_node_cont} for select MN counties using traffic data to estimate inter-county travel flows (dashed lines) with best-fit parameters learned from the estimated epidemic states using COVID-19 infection data (solid lines).}
    \label{fig:param_fit_flowsmodel}
\end{figure}

\begin{figure}
    \centering
    \captionsetup{aboveskip=-1pt,belowskip=-12pt}
    \includegraphics[width=.7\columnwidth]{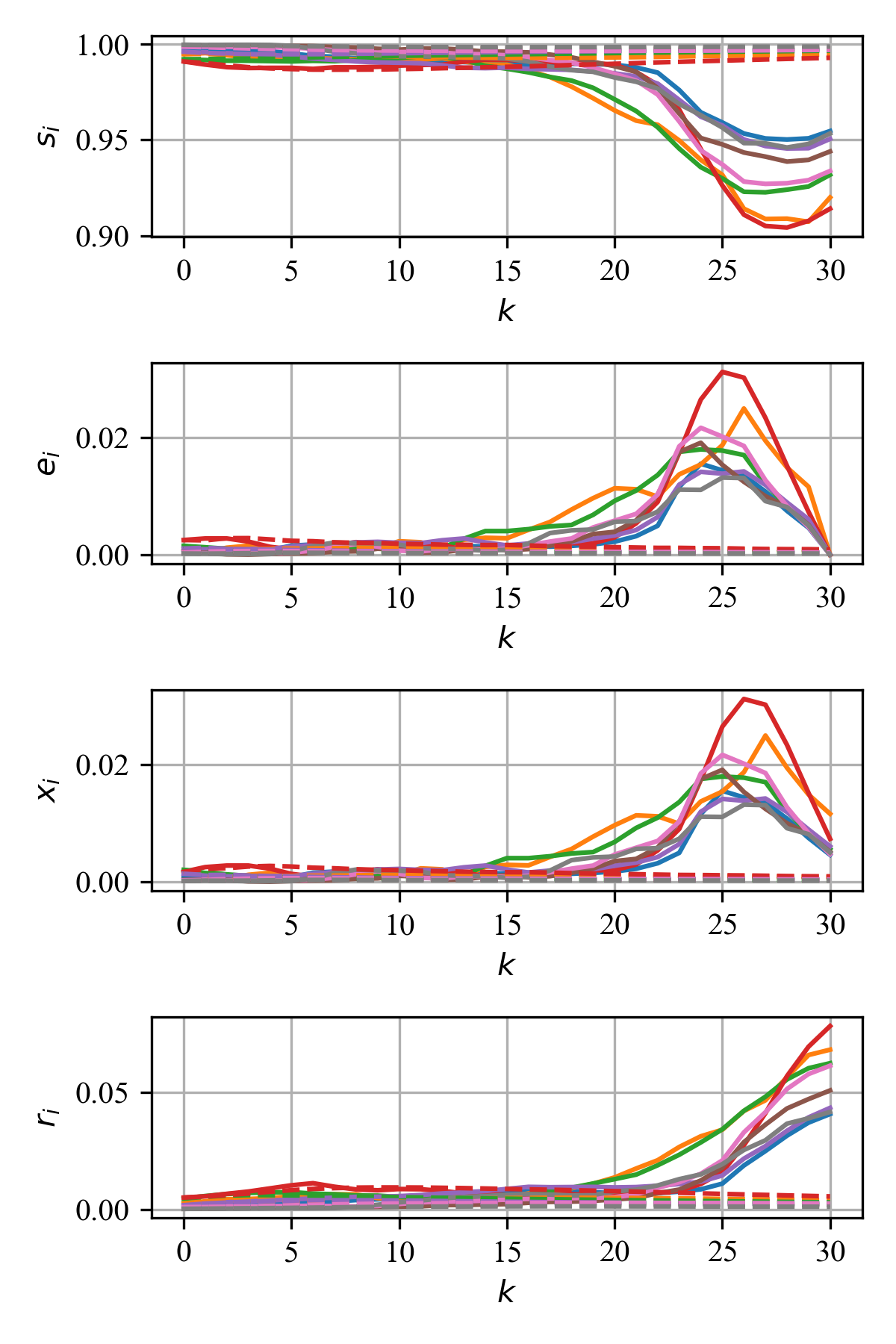}
    \caption{Simulating the modified direct connection model from \cite{vrabac2021capturing} for select MN counties using traffic data to estimate inter-county travel flows (dashed lines) with best-fit parameters learned from the estimated epidemic states using COVID-19 infection data (solid lines).}
    \label{fig:param_fit_LCSSmodel}
\end{figure}

\subsection{Predicting Disease Arrival Time}
We use the methods described in Section~\ref{sec:eff_dist_flows} and the 
collected travel flows data to make predictions on the arrival time of COVID-19 to counties in MN. In Figure~\ref{fig:MN_arrival_vs_dist} we compare the arrival time of the first reported case with both the effective distance and geographic distance of each county to the disease origin (i.e., the county with the first reported case in MN), with the best-fit line shown for each. We compute the R-value for each best-fit line to be $0.53$ and $0.48$ for the effective distance and geographic distance, respectively, suggesting that the effective distance is a better predictor of arrival time, albeit by a small margin, than geographic distance when considering all arrival times with respect to the disease origin. 
Finally, we demonstrate the shifting window prediction method on the MN COVID-19 data in Figure~\ref{fig:MN_k_day_predict}. We find an RMS error of close to 19 days for the linear fit of all arrival times versus effective distance from the disease origin versus an average RMS error of 4 days for the shifting window predicting the next 10 arrival times whenever a new arrival occurs, yielding a significant $79\%$ reduction in prediction error in addition to enabling real-time arrival predictions. 

\begin{figure}
    \centering
    \includegraphics[width=.6\columnwidth]{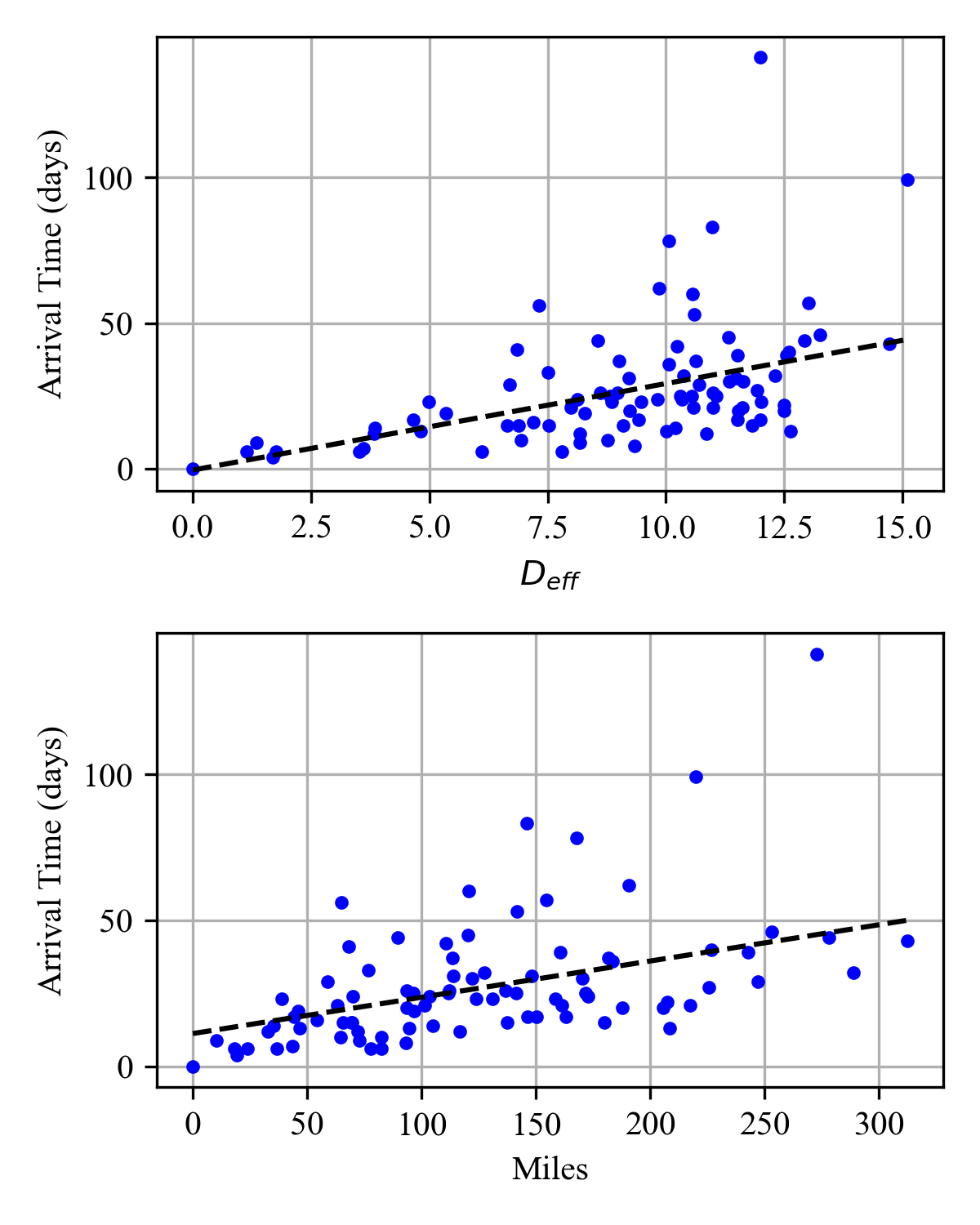}
    \caption{The arrival time of the disease to each node versus the effective distance (top) of each county in MN from the county with the first detected COVID-19 cases during the 2020 pandemic, computed by \eqref{eq:eff_dist_origin}, compared with the geographic distance of each county (bottom) from the disease origin. The dashed line shows the linear fit of all data points for each case respectively.}
    \label{fig:MN_arrival_vs_dist}
\end{figure}

\begin{figure}
\centering
\begin{subfigure}{0.45\columnwidth}
    \includegraphics[width=\textwidth]{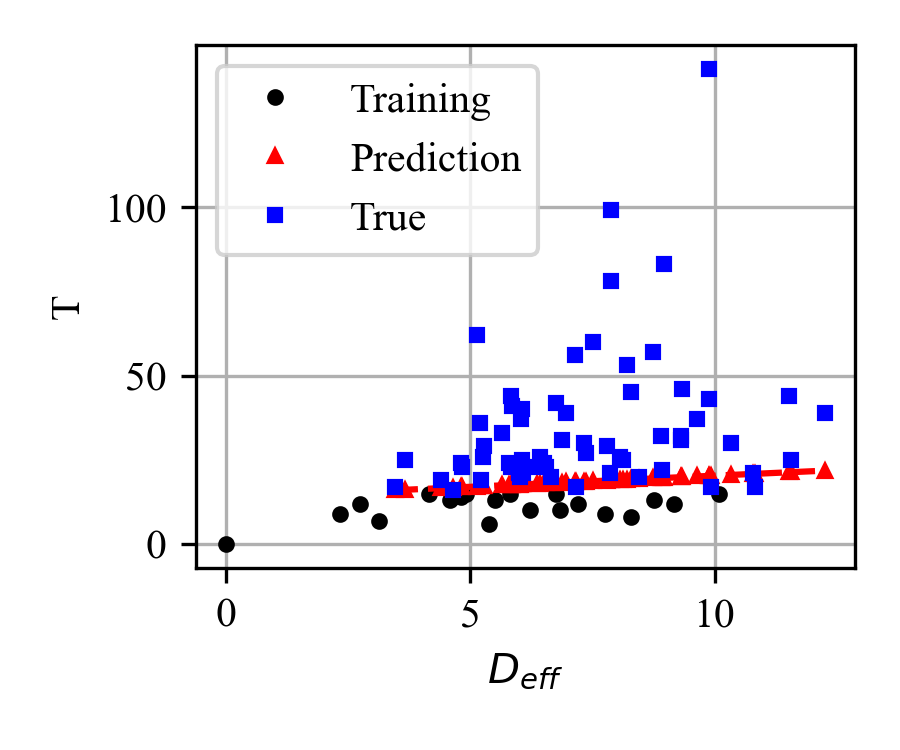}
    \caption{$T=T_{45}$}
\end{subfigure}
\begin{subfigure}{0.45\columnwidth}
    \includegraphics[width=\textwidth]{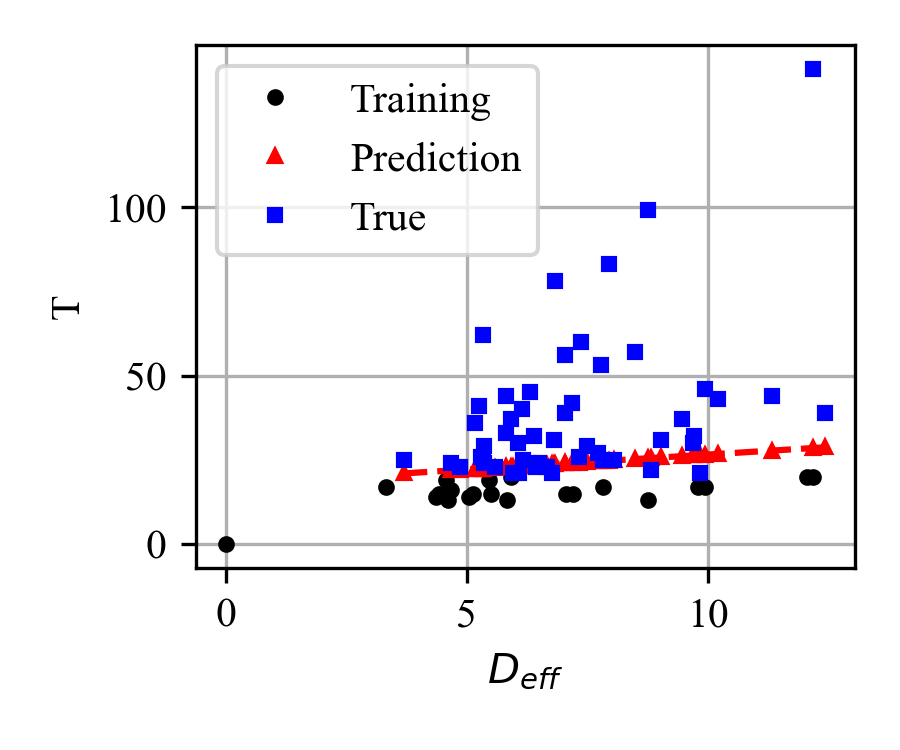}
    \caption{$T=T_{55}$} 
\end{subfigure} 
\begin{subfigure}{0.45\columnwidth} 
    \includegraphics[width=\textwidth]{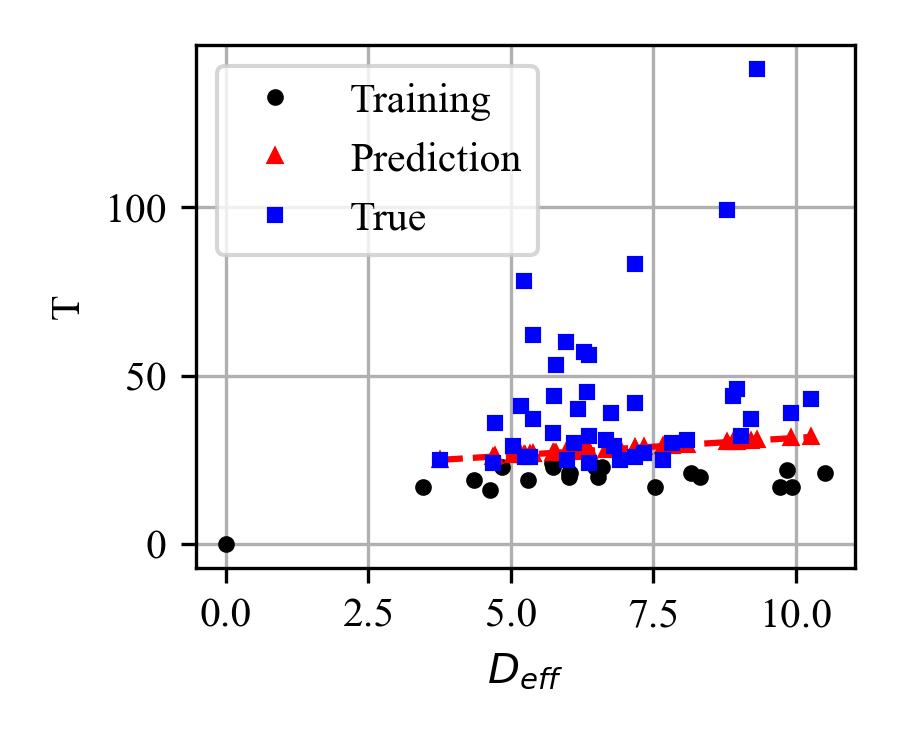} 
    \caption{$T=T_{65}$} 
\end{subfigure}  
\begin{subfigure}{0.45\columnwidth} 
    \includegraphics[width=\textwidth]{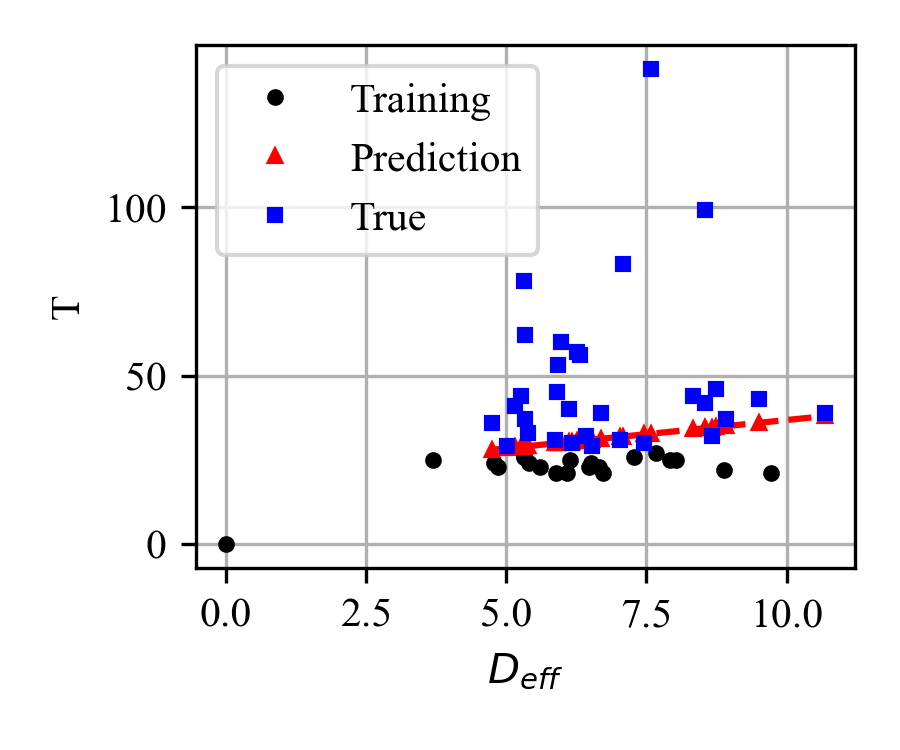} 
    \caption{$T=T_{75}$} 
\end{subfigure}
\caption{Examples of the shifting prediction window of arrival times of COVID-19 for counties in MN at four different arrival times during the pandemic ($T=T_{45},T_{75}$) where training window $\tau=20$, with the data used to make predictions (black circles), true arrival times for each county versus effective distance (blue squares), and shifted best-fit predictions (red triangles) shown at several arrival intervals, where the effective distances is computed at each disease arrival using \eqref{eq:eff_dist_Xcal}.}
\label{fig:MN_k_day_predict}
\end{figure}

\section{Conclusion}

In this work, we have developed a model for simulating the spread of an SEIRS epidemic process using networked population flows as the propagation mechanism. We provide sufficient conditions under which the healthy state of the system will be locally stable or unstable and show via model analysis that there exists no valid perturbation to the population flows that will change the local stability of any healthy state. These results suggest that manipulating population flow alone is insufficient to achieve the long-term eradication of the disease in the defined system where a loss of immunity is present. Conversely, the result also suggests that perturbing the population flows will not prevent a system from achieving disease eradication (i.e., causing the healthy state to become unstable if it is already locally stable). Further, we provide sufficient conditions under which the system will enter a unique endemic state. Although the analytical results show that travel control is insufficient for affecting the long-term behavior of the epidemic, we show that travel flows are a good predictor of disease arrival time, which can provide valuable model-agnostic estimates of the arrival times for a given disease that are useful in allocating a fixed amount of medical resources. To this end, we then develop tools for applying our proposed model to data, such as spreading parameter identification and disease arrival time prediction and illustrate these tools using a case study of both travel and infection data from the counties in MN during part of the COVID-19 pandemic.

\edit{
There are several directions for continued research on the work presented in this paper. Future directions for analysis included a study on the global stability of the healthy state, as only local stability is explored in this work, as well as a study of the stability of the endemic equilibrium for this system. Additionally, in this work, we have assumed that no infections occur during the transit of individuals between populations, and adding this element to our travel flows model may provide additional meaningful insights into the role of travel in spreading infectious diseases. Further, since the use of effective distance to predict disease arrival time is independent of the epidemic model being used, a more thorough analysis can be performed comparing the usefulness of effective distance in predicting arrival time for other epidemic models with travel flows (such as networked SIR, SIS, SAIR, etc.).
}

\bibliographystyle{unsrtnat}
\bibliography{references}  

\begin{thebibliography}{50}
\providecommand{\natexlab}[1]{#1}
\providecommand{\url}[1]{\texttt{#1}}
\expandafter\ifx\csname urlstyle\endcsname\relax
  \providecommand{\doi}[1]{doi: #1}\else
  \providecommand{\doi}{doi: \begingroup \urlstyle{rm}\Url}\fi

\bibitem[Ruan et~al.(2006)Ruan, Wang, and Levin]{ruan2006effect}
Shigui Ruan, Wendi Wang, and Simon~A Levin.
\newblock The effect of global travel on the spread of {SARS}.
\newblock \emph{Mathematical Biosciences \& Engineering}, 3\penalty0 (1):\penalty0 205, 2006.

\bibitem[Tatem et~al.(2006)Tatem, Rogers, and Hay]{tatem2006global}
Andrew~J Tatem, David~J Rogers, and Simon~I Hay.
\newblock Global transport networks and infectious disease spread.
\newblock \emph{Advances in Parasitology}, 62:\penalty0 293--343, 2006.

\bibitem[Lowe et~al.(2014)Lowe, Gauger, Harmon, Zhang, Connor, Yeske, Loula, Levis, Dufresne, and Main]{lowe2014role}
James Lowe, Phillip Gauger, Karen Harmon, Jianqiang Zhang, Joseph Connor, Paul Yeske, Timothy Loula, Ian Levis, Luc Dufresne, and Rodger Main.
\newblock Role of transportation in spread of porcine epidemic diarrhea virus infection, united states.
\newblock \emph{Emerging Infectious Diseases}, 20\penalty0 (5):\penalty0 872, 2014.

\bibitem[Findlater and Bogoch(2018)]{findlater2018human}
Aidan Findlater and Isaac~I Bogoch.
\newblock Human mobility and the global spread of infectious diseases: a focus on air travel.
\newblock \emph{Trends in Parasitology}, 34\penalty0 (9):\penalty0 772--783, 2018.

\bibitem[Hollingsworth et~al.(2006)Hollingsworth, Ferguson, and Anderson]{hollingsworth2006will}
T~D{\'e}irdre Hollingsworth, Neil~M Ferguson, and Roy~M Anderson.
\newblock Will travel restrictions control the international spread of pandemic influenza?
\newblock \emph{Nature Medicine}, 12\penalty0 (5):\penalty0 497--499, 2006.

\bibitem[Tang et~al.(2009)Tang, Liu, and Li]{tang2009epidemic}
Ming Tang, Zonghua Liu, and Baowen Li.
\newblock Epidemic spreading by objective traveling.
\newblock \emph{EPL (Europhysics Letters)}, 87\penalty0 (1):\penalty0 18005, 2009.

\bibitem[Maital and Barzani(2020)]{maital2020global}
Shlomo Maital and Ella Barzani.
\newblock The global economic impact of {COVID-19}: A summary of research.
\newblock \emph{Samuel Neaman Institute for National Policy Research}, 2020:\penalty0 1--12, 2020.

\bibitem[Kaye et~al.(2021)Kaye, Okeagu, Pham, Silva, Hurley, Arron, Sarfraz, Lee, Ghali, Gamble, et~al.]{kaye2021economic}
Alan~D Kaye, Chikezie~N Okeagu, Alex~D Pham, Rayce~A Silva, Joshua~J Hurley, Brett~L Arron, Noeen Sarfraz, Hong~N Lee, Ghali~E Ghali, Jack~W Gamble, et~al.
\newblock Economic impact of {COVID-19} pandemic on healthcare facilities and systems: International perspectives.
\newblock \emph{Best Practice \& Research Clinical Anaesthesiology}, 35\penalty0 (3):\penalty0 293--306, 2021.

\bibitem[Ahmad et~al.(2020)Ahmad, Baig, and Hui]{ahmad2020coronavirus}
Tauseef Ahmad, Mukhtiar Baig, and Jin Hui.
\newblock Coronavirus disease 2019 ({COVID}-19) pandemic and economic impact.
\newblock \emph{Pakistan journal of medical sciences}, 36\penalty0 (COVID19-S4):\penalty0 S73, 2020.

\bibitem[Kumar et~al.(2020)Kumar, Maheshwari, Prabhu, Prasanna, Jayalakshmi, Suganya, Malar, and Jothikumar]{kumar2020social}
Sandeep Kumar, V~Maheshwari, J~Prabhu, M~Prasanna, P~Jayalakshmi, P~Suganya, Benjula~Anbu Malar, and R~Jothikumar.
\newblock Social economic impact of {COVID}-19 outbreak in {I}ndia.
\newblock \emph{International Journal of Pervasive Computing and Communications}, 2020.

\bibitem[Kermack and McKendrick(1927)]{kermack1927contribution}
William~Ogilvy Kermack and Anderson~G McKendrick.
\newblock A contribution to the mathematical theory of epidemics.
\newblock \emph{Proceedings of the Royal Society of London. Series A, Containing Papers of a Mathematical and Physical Character}, 115\penalty0 (772):\penalty0 700--721, 1927.

\bibitem[Brauer(2008)]{brauer2008compartmental}
Fred Brauer.
\newblock \emph{Compartmental Models in Epidemiology}, pages 19--79.
\newblock Springer Berlin Heidelberg, Berlin, Heidelberg, 2008.

\bibitem[Prasse and Van~Mieghem(2020)]{prasse2020network}
Bastian Prasse and Piet Van~Mieghem.
\newblock Network reconstruction and prediction of epidemic outbreaks for general group-based compartmental epidemic models.
\newblock \emph{IEEE Transactions on Network Science and Engineering}, 7\penalty0 (4):\penalty0 2755--2764, 2020.

\bibitem[Fagnani and Zino(2017)]{fagnani2017time}
Fabio Fagnani and Lorenzo Zino.
\newblock Time to extinction for the {SIS} epidemic model: New bounds on the tail probabilities.
\newblock \emph{IEEE Transactions on Network Science and Engineering}, 6\penalty0 (1):\penalty0 74--81, 2017.

\bibitem[Chen et~al.(2019)Chen, Small, and Fu]{chen2019global}
Shanshan Chen, Michael Small, and Xinchu Fu.
\newblock Global stability of epidemic models with imperfect vaccination and quarantine on scale-free networks.
\newblock \emph{IEEE Transactions on Network Science and Engineering}, 7\penalty0 (3):\penalty0 1583--1596, 2019.

\bibitem[Xiang et~al.(2021)Xiang, Jia, Chen, Guo, Shu, and Long]{xiang2021covid}
Yue Xiang, Yonghong Jia, Linlin Chen, Lei Guo, Bizhen Shu, and Enshen Long.
\newblock {COVID}-19 epidemic prediction and the impact of public health interventions: A review of {COVID}-19 epidemic models.
\newblock \emph{Infectious Disease Modelling}, 6:\penalty0 324--342, 2021.

\bibitem[Rahmandad et~al.(2021)Rahmandad, Lim, and Sterman]{rahmandad2021behavioral}
Hazhir Rahmandad, Tse~Yang Lim, and John Sterman.
\newblock Behavioral dynamics of {COVID}-19: Estimating underreporting, multiple waves, and adherence fatigue across 92 nations.
\newblock \emph{System Dynamics Review}, 37\penalty0 (1):\penalty0 5--31, 2021.

\bibitem[Wang(2020)]{wang2020evaluating}
Yuanhao Wang.
\newblock Evaluating the effectiveness of {COVID}-19 prevention and control measures based on {SEIR} model.
\newblock In \emph{Proceedings of the 2020 International Conference on Communications, Information System and Computer Engineering (CISCE)}, pages 143--146. IEEE, 2020.

\bibitem[Chen et~al.(2020)Chen, Lu, Chang, and Liu]{chen2020time}
Yi-Cheng Chen, Ping-En Lu, Cheng-Shang Chang, and Tzu-Hsuan Liu.
\newblock A time-dependent {SIR} model for {COVID}-19 with undetectable infected persons.
\newblock \emph{IEEE Transactions on Network Science and Engineering}, 7\penalty0 (4):\penalty0 3279--3294, 2020.

\bibitem[Nowzari et~al.(2016)Nowzari, Preciado, and Pappas]{nowzari2016analysis}
Cameron Nowzari, Victor~M Preciado, and George~J Pappas.
\newblock Analysis and control of epidemics: A survey of spreading processes on complex networks.
\newblock \emph{IEEE Control Systems Magazine}, 36\penalty0 (1):\penalty0 26--46, 2016.

\bibitem[Qian and Ukkusuri(2021)]{qian2021connecting}
Xinwu Qian and Satish~V Ukkusuri.
\newblock Connecting urban transportation systems with the spread of infectious diseases: A {T}rans-{SEIR} modeling approach.
\newblock \emph{Transportation Research Part B: Methodological}, 145:\penalty0 185--211, 2021.

\bibitem[Salehi et~al.(2015)Salehi, Sharma, Marzolla, Magnani, Siyari, and Montesi]{salehi2015spreading}
Mostafa Salehi, Rajesh Sharma, Moreno Marzolla, Matteo Magnani, Payam Siyari, and Danilo Montesi.
\newblock Spreading processes in multilayer networks.
\newblock \emph{IEEE Transactions on Network Science and Engineering}, 2\penalty0 (2):\penalty0 65--83, 2015.

\bibitem[Qu and Wang(2017)]{qu2017sis}
Bo~Qu and Huijuan Wang.
\newblock {SIS} epidemic spreading with heterogeneous infection rates.
\newblock \emph{IEEE Transactions on Network Science and Engineering}, 4\penalty0 (3):\penalty0 177--186, 2017.

\bibitem[Kang et~al.(2018)Kang, Sun, Yu, Fu, and Bao]{kang2018spreading}
Huiyan Kang, Mengfeng Sun, Yajuan Yu, Xinchu Fu, and Bocheng Bao.
\newblock Spreading dynamics of an {SEIR} model with delay on scale-free networks.
\newblock \emph{IEEE Transactions on Network Science and Engineering}, 7\penalty0 (1):\penalty0 489--496, 2018.

\bibitem[Nadini et~al.(2018)Nadini, Rizzo, and Porfiri]{nadini2018epidemic}
Matthieu Nadini, Alessandro Rizzo, and Maurizio Porfiri.
\newblock Epidemic spreading in temporal and adaptive networks with static backbone.
\newblock \emph{IEEE Transactions on Network Science and Engineering}, 7\penalty0 (1):\penalty0 549--561, 2018.

\bibitem[Li et~al.(2021)Li, Xiang, and He]{li2021modeling}
Jian Li, Tao Xiang, and Linghui He.
\newblock Modeling epidemic spread in transportation networks: A review.
\newblock \emph{Journal of Traffic and Transportation Engineering (English Edition)}, 8\penalty0 (2):\penalty0 139--152, 2021.

\bibitem[Yang et~al.(2012)Yang, Wang, Chen, and Wang]{yang2012epidemic}
Xu-Hua Yang, Bo~Wang, Sheng-Yong Chen, and Wan-Liang Wang.
\newblock Epidemic dynamics behavior in some bus transport networks.
\newblock \emph{Physica A: Statistical Mechanics and its Applications}, 391\penalty0 (3):\penalty0 917--924, 2012.

\bibitem[Hackl and Dubernet(2019)]{hackl2019epidemic}
J{\"u}rgen Hackl and Thibaut Dubernet.
\newblock Epidemic spreading in urban areas using agent-based transportation models.
\newblock \emph{Future Internet}, 11\penalty0 (4):\penalty0 92, 2019.

\bibitem[Par{\'e} et~al.(2020)Par{\'e}, Beck, and Ba{\c{s}}ar]{pare2020modeling}
Philip~E Par{\'e}, Carolyn~L Beck, and Tamer Ba{\c{s}}ar.
\newblock Modeling, estimation, and analysis of epidemics over networks: An overview.
\newblock \emph{Annual Reviews in Control}, 50:\penalty0 345--360, 2020.

\bibitem[Ye et~al.(2020)Ye, Liu, Cenedese, Sun, and Cao]{ye2020network}
Mengbin Ye, Ji~Liu, Carlo Cenedese, Zhiyong Sun, and Ming Cao.
\newblock A network {SIS} meta-population model with transportation flow.
\newblock In \emph{Proceedings of the IFAC World Congress}, pages 2562--2567, 2020.

\bibitem[Yin et~al.(2020)Yin, Wang, Xia, Dehmer, Emmert-Streib, and Jin]{yin2020novel}
Qian Yin, Zhishuang Wang, Chengyi Xia, Matthias Dehmer, Frank Emmert-Streib, and Zhen Jin.
\newblock A novel epidemic model considering demographics and intercity commuting on complex dynamical networks.
\newblock \emph{Applied Mathematics and Computation}, 386:\penalty0 125517, 2020.

\bibitem[Brockmann and Helbing(2013)]{brockmann2013hidden}
Dirk Brockmann and Dirk Helbing.
\newblock The hidden geometry of complex, network-driven contagion phenomena.
\newblock \emph{Science}, 342\penalty0 (6164):\penalty0 1337--1342, 2013.

\bibitem[Di~Giamberardino et~al.(2021)Di~Giamberardino, Iacoviello, Papa, and Sinisgalli]{di2021data}
Paolo Di~Giamberardino, Daniela Iacoviello, Federico Papa, and Carmela Sinisgalli.
\newblock A data-driven model of the {COVID-19} spread among interconnected populations: Epidemiological and mobility aspects following the lockdown in {I}taly.
\newblock \emph{Nonlinear Dynamics}, 106:\penalty0 1239--1266, 2021.

\bibitem[Arino and van~den Driessche(2003)]{arino2003multi}
Julien Arino and Pauline van~den Driessche.
\newblock A multi-city epidemic model.
\newblock \emph{Mathematical Population Studies}, 10\penalty0 (3):\penalty0 175--193, 2003.

\bibitem[Liu and Stechlinski(2013)]{liu2013transmission}
Xinzhi Liu and Peter Stechlinski.
\newblock Transmission dynamics of a switched multi-city model with transport-related infections.
\newblock \emph{Nonlinear Analysis: Real World Applications}, 14\penalty0 (1):\penalty0 264--279, 2013.

\bibitem[Knipl(2016)]{knipl2016stability}
Di{\'a}na Knipl.
\newblock Stability criteria for a multi-city epidemic model with travel delays and infection during travel.
\newblock \emph{Electronic Journal of Qualitative Theory of Differential Equations}, 2016\penalty0 (74):\penalty0 1--22, 2016.

\bibitem[Pujari and Shekatkar(2020)]{pujari2020multi}
Bhalchandra~S Pujari and Snehal Shekatkar.
\newblock Multi-city modeling of epidemics using spatial networks: Application to 2019-{nCov} ({COVID-19}) coronavirus in {India}.
\newblock \emph{Medrxiv}, pages 2020--03, 2020.

\bibitem[Arino and Van~den Driessche(2006)]{arino2006disease}
Julien Arino and P~Van~den Driessche.
\newblock Disease spread in metapopulations.
\newblock \emph{Fields Institute Communications}, 48\penalty0 (1):\penalty0 1--13, 2006.

\bibitem[Wang and Li(2014)]{wang2014spatial}
Lin Wang and Xiang Li.
\newblock Spatial epidemiology of networked metapopulation: An overview.
\newblock \emph{Chinese Science Bulletin}, 59\penalty0 (28):\penalty0 3511--3522, 2014.

\bibitem[Shao and Han(2022)]{shao2022epidemic}
Qi~Shao and Dun Han.
\newblock Epidemic spreading in metapopulation networks with heterogeneous mobility rates.
\newblock \emph{Applied Mathematics and Computation}, 412:\penalty0 126559, 2022.

\bibitem[Zhu et~al.(2021)Zhu, Liu, Wang, Wang, Chen, and Wang]{zhu2021allocating}
Xuzhen Zhu, Yuxin Liu, Shengfeng Wang, Ruijie Wang, Xiaolong Chen, and Wei Wang.
\newblock Allocating resources for epidemic spreading on metapopulation networks.
\newblock \emph{Applied Mathematics and Computation}, 411:\penalty0 126531, 2021.

\bibitem[Vrabac et~al.(2022)Vrabac, Shang, Butler, Pham, Stern, and Par{\'e}]{vrabac2021capturing}
Damir Vrabac, Mingfeng Shang, Brooks Butler, Joseph Pham, Raphael Stern, and Philip~E Par{\'e}.
\newblock Capturing the effects of transportation on the spread of {{COVID-19}} with a multi-networked {SEIR} model.
\newblock \emph{IEEE Control Systems Letters}, 6:\penalty0 103--108, 2022.

\bibitem[Butler et~al.(2021)Butler, Zhang, Walter, Nair, Stern, and Par{\'e}]{butler2021effect}
Brooks Butler, Ciyuan Zhang, Ian Walter, Nishant Nair, Raphael Stern, and Philip~E Par{\'e}.
\newblock The effect of population flow on epidemic spread: Analysis and control.
\newblock In \emph{Proceedings of the 2021 60th IEEE Conference on Decision and Control (CDC)}, pages 4260--4265, 2021.

\bibitem[Varga(1962)]{varga1962iterative}
Richard~S Varga.
\newblock \emph{Iterative Analysis}.
\newblock Springer, 1962.

\bibitem[Horn and Johnson(2012)]{horn2012matrix}
Roger~A Horn and Charles~R Johnson.
\newblock \emph{Matrix Analysis}.
\newblock Cambridge University Press, 2012.

\bibitem[Shu et~al.(2012)Shu, Fan, and Wei]{shu2012global}
Hongying Shu, Dejun Fan, and Junjie Wei.
\newblock Global stability of multi-group {SEIR} epidemic models with distributed delays and nonlinear transmission.
\newblock \emph{Nonlinear Analysis: Real World Applications}, 13\penalty0 (4):\penalty0 1581--1592, 2012.

\bibitem[Hu et~al.(2021)]{hu2021flights}
Maogui Hu et~al.
\newblock {Risk of Severe Acute Respiratory Syndrome Coronavirus 2 (SARS-CoV-2) Transmission Among Air Passengers in China}.
\newblock \emph{Clinical Infectious Diseases}, 75\penalty0 (1):\penalty0 e234--e240, 09 2021.

\bibitem[Butler(2023)]{butlercode2023}
Brooks~A. Butler.
\newblock Networked {SEIRS} with travel flows.
\newblock \url{https://github.com/brooksbutler/SEIRS_PopulationFlows}, 2023.

\bibitem[MNd()]{MNdata}
Minnesota coronavirus cases and deaths.
\newblock \url{https://usafacts.org/visualizations/coronavirus-{COVID-19}-spread-map/state/minnesota}.
\newblock {Accessed} September 16, 2022.

\bibitem[str()]{streetlight}
{Street Light Data}.
\newblock \url{https://www.streetlightdata.com/}.
\newblock {Accessed} November 2, 2022.

\end{thebibliography}






\end{document}